\documentclass[numbook,envcountsame,draft]{svjour3}
\pdfoutput=1
\usepackage{enumitem}
 \makeatletter
 \if@twocolumn
   \renewcommand\normalsize{%
    \@setfontsize\normalsize\@xpt{12.5pt}%
    \abovedisplayskip=3 mm plus6pt minus 4pt
    \belowdisplayskip=3 mm plus6pt minus 4pt
    \abovedisplayshortskip=0.0 mm plus6pt
    \belowdisplayshortskip=2 mm plus4pt minus 4pt
    \let\@listi\@listI}%

   \renewcommand\small{%
    \@setfontsize\small{8.5pt}\@xpt
    \abovedisplayskip 8.5\p@ \@plus3\p@ \@minus4\p@
    \abovedisplayshortskip \z@ \@plus2\p@
    \belowdisplayshortskip 4\p@ \@plus2\p@ \@minus2\p@
    \def\@listi{\leftmargin\leftmargini
                \parsep 0\p@ \@plus1\p@ \@minus\p@
                \topsep 4\p@ \@plus2\p@ \@minus4\p@
                \itemsep0\p@}%
    \belowdisplayskip \abovedisplayskip}

 \else
   \if@smallext
    \renewcommand\normalsize{%
    \@setfontsize\normalsize\@xpt\@xiipt
    \abovedisplayskip=3 mm plus6pt minus 4pt
    \belowdisplayskip=3 mm plus6pt minus 4pt
    \abovedisplayshortskip=0.0 mm plus6pt
    \belowdisplayshortskip=2 mm plus4pt minus 4pt
    \let\@listi\@listI}%

   \renewcommand\small{%
    \@setfontsize\small\@viiipt{9.5pt}%
    \abovedisplayskip 8.5\p@ \@plus3\p@ \@minus4\p@
    \abovedisplayshortskip \z@ \@plus2\p@
    \belowdisplayshortskip 4\p@ \@plus2\p@ \@minus2\p@
    \def\@listi{\leftmargin\leftmargini
                \parsep 0\p@ \@plus1\p@ \@minus\p@
                \topsep 4\p@ \@plus2\p@ \@minus4\p@
                \itemsep0\p@}%
    \belowdisplayskip \abovedisplayskip}
  \else
   \renewcommand\normalsize{%
    \@setfontsize\normalsize{9.5pt}{11.5pt}%
    \abovedisplayskip=3 mm plus6pt minus 4pt
    \belowdisplayskip=3 mm plus6pt minus 4pt
    \abovedisplayshortskip=0.0 mm plus6pt
    \belowdisplayshortskip=2 mm plus4pt minus 4pt
    \let\@listi\@listI}%

   \renewcommand\small{%
    \@setfontsize\small\@viiipt{9.25pt}%
    \abovedisplayskip 8.5\p@ \@plus3\p@ \@minus4\p@
    \abovedisplayshortskip \z@ \@plus2\p@
    \belowdisplayshortskip 4\p@ \@plus2\p@ \@minus2\p@
    \def\@listi{\leftmargin\leftmargini
                \parsep 0\p@ \@plus1\p@ \@minus\p@
                \topsep 4\p@ \@plus2\p@ \@minus4\p@
                \itemsep0\p@}%
    \belowdisplayskip \abovedisplayskip}
   \fi
 \fi
 
 \makeatother
\smartqed
\usepackage[final]{microtype}
\usepackage{mathptmx}
\DeclareMathAlphabet{\mathcal}{OMS}{cmsy}{m}{n}
\usepackage{amsmath}
\usepackage{mathrsfs}
\usepackage{amssymb}
\usepackage{tikz}
\newcommand{\vtms}{\mkern0.3\thinmuskip}

\newcommand{\coeff}{\mathrm{coeff}}
\newcommand{\opp}{^{\mathrm{o}}}
\newcommand{\lside}{_{\mathsf{L}}}
\newcommand{\rside}{_{\mathsf{R}}}

\newcommand{\powerset}{\mathcal{P}}

\DeclareMathOperator{\SD}{\mathrm{SD}}
\DeclareMathOperator{\SA}{\mathrm{SA}}
\DeclareMathOperator{\WD}{\mathrm{WD}}
\DeclareMathOperator{\WA}{\mathrm{WA}}
\DeclareMathOperator{\D}{\mathrm{D}}
\DeclareMathOperator{\A}{\mathrm{A}}
\DeclareMathOperator{\Pos}{\mathrm{Pos}}
\spnewtheorem*{notation}{Notation.}{\it}{\rm}

\begin{document}
\title{\textit{W-}graph ideals and biideals}\label{Pre}
\author{Robert B. Howlett \and Van Minh Nguyen}
\institute{   Van Minh Nguyen \at
              University of Sydney \\
              Fax: +61 2 9351 4534\\
              \email{VanNguyen@maths.usyd.edu.au}
}
\date{}
\let\makeheadbox\relax
\maketitle

\setitemindent{xx(iii)}
\begin{abstract}
We further develop the theory of \(W\!\)-graph ideals, first introduced
in~\cite{howvan:wgraphDetSets}. We discuss \(W\!\)-graph subideals, and
induction and restriction of \(W\!\)-graph ideals for parabolic
subgroups. We introduce \(W\!\)-graph biideals: those \(W\!\)-graph
ideals that yield \((W\times W^{\mathrm o})\)-graphs, where \(W^{\mathrm o}\)
is the group opposite to~\(W\). We determine all \(W\!\)-graph ideals and
biideals in finite Coxeter groups of rank~2.

\keywords{Coxeter groups \and Hecke algebras \and \(W\!\)-graphs \and
Kazhdan--Lusztig polynomials \and cells}
\end{abstract}

\section{Introduction}
\label{intro}
Let \((W,S)\) be a Coxeter system and \(\mathcal{H}(W)\) its Hecke
algebra over \(\mathbb{Z}[q,q^{-1}]\), the ring of Laurent
polynomials in the indeterminate \(q\). The Coxeter system
\((W,S)\) is naturally equipped with the left weak order and
the Bruhat order, denoted by \(\leqslant\lside\) and \(\leqslant\), respectively.
In~\cite{howvan:wgraphDetSets}, an algorithm was given to produce
from an ideal (down set) \(\mathscr I\) of \((W,\leqslant\lside)\) and a subset
\(J\) of \(S \setminus \mathscr I\) a weighted digraph \(\Gamma(\mathscr I\!,\,J)\) with
vertices indexed by the elements of \(\mathscr I\) and coloured with
subsets of \(S\). If, in the terminology of \cite{howvan:wgraphDetSets},
\(\mathscr I\) is a \(W\!\)-graph ideal with respect to \(J\), then
\(\Gamma(\mathscr I\!,\,J)\) is a \(W\!\)-graph. In the present paper we use the
terminology ``\((\mathscr I\!,\,J)\) is a \(W\!\)-graph ideal'' to mean the same
thing as ``\(\mathscr I\) is a \(W\!\)-graph ideal with respect to \(J\)\,''.

The algorithm referred to above proceeds chiefly by recursively
computing polynomials \(q_{y,w}\) for all \(y,w \in \mathscr{I}\) such that
\(y < w\). These polynomials are anologous to Kazhdan--Lusztig polynomials, and the
Kazhdan--Lusztig \(W\!\)-graph (\cite{kazlus:coxhecke}) and Deodhar's parabolic
analogues (\cite{deo:paraKL}) are obtained as special cases. Moreover, it was shown
in~\cite{nguyen:wgideals2} that \(W\!\)-graphs for the Kazhdan--Lusztig left cells
that contain longest elements of standard parabolic subgroups can be constructed this way.
In type \(A\), this provides a practical procedure for calculating a \(W\!\)-graph
for a cell module (which is known to be isomorphic to the corresponding Specht module) from
standard tableaux of a given shape.

In general, it is still unknown which subsets of \(W\) generate \(W\!\)-graph ideals, and
the problem of describing them combinatorially is still open, even in type~\(A\). Preliminary
results concerning these matters in type~\(A\) are established
in~\cite{nguyen:wgdeterminingelmA}, using the results of the present paper combined with
those of~\cite{howvan:wgraphDetSets,nguyen:wgideals2}.

In this paper, we define a \(W\!\)-graph subideal of a \(W\!\)-graph ideal
\((\mathscr{I}\!,\,J)\) to be a \(W\!\)-graph ideal \((\mathscr{L},K)\) such
that \(\mathscr{L}\subseteq\mathscr{I}\) and \(K=J\).
It was shown in~\cite{nguyen:wgideals2} that if \((\mathscr{I}\!,\,J)\) is a
\(W\!\)-graph ideal and \(\mathscr{L}\subseteq \mathscr{I}\!\) then
\((\mathscr{L},J)\) is a \(W\!\)-graph subideal of \((\mathscr{I}\!,\,J)\)
if the complement \(\mathscr{I}\setminus\mathscr{L}\) is closed when regarded
as a subset of the vertex set of \(\Gamma=\Gamma(\mathscr{I}\!,\,J)\)
(in the sense that it is an ideal with respect to the Kazhdan--Lusztig
preorder~\(\leqslant_\Gamma\) on the vertex set). We call
\(W\!\)-graph subideals of this form \textit{strong\/} \(W\!\)-graph subideals.
We show that this strong \(W\!\)-graph subideal relation is preserved by induction of
\(W\!\)-graph ideals, as defined in~\cite[Section 9]{howvan:wgraphDetSets}. More precisely,
if \(W_K\) is a standard parabolic subgroup of~\(W\) (where \(K\subseteq S\)),
and \(D_K\) denotes the set of minimal length representatives of left cosets of
\(W_K\) in~\(W\), then \((D_K\mathscr{L},J)\) is a strong \(W\!\)-graph subideal
of \((D_K\mathscr{I}\!,\,J)\) if \((\mathscr{L},J)\) is a strong \(W_K\)-graph
subideal of \((\mathscr{I}\!,\,J)\).

Recall that the original construction given by Kazhdan and Lusztig
in~\cite{kazlus:coxhecke} produces a \((W\times W^{\mathrm o})\)-graph,
where \(W^{\mathrm o}\) is the Coxeter group opposite to \(W\). Thus it is natural
to seek a generalization the results of~\cite{howvan:wgraphDetSets} that
produces \((W\times W^{\mathrm o})\)-graphs. This is the motivation for the
\(W\!\)-graph biideal concept.

As mentioned earlier, for an arbitrary Coxeter system \((W, S)\), the algorithm
in~\cite{howvan:wgraphDetSets} takes as
input an ideal \(\mathscr I\) of \((W,\leqslant\lside)\) and a subset
\(J\) of \(S \setminus \mathscr I\), and produces a (decorated) graph
\(\Gamma(\mathscr I\!,\,J)\) as
output. If \(\mathscr I\) is a \(W\!\)-graph ideal with respect to \(J\), then
\(\Gamma(\mathscr I\!,\,J)\) is \(W\!\)-graph. It is natural to ask whether this
condition characterizes
\(W\!\)-graph ideals. The answer to this question is affirmative: \(W\!\)-graph ideals
are precisely the ideals for which the above construction produces \(W\!\)-graphs.
This is useful in practice as a computational means of determining whether or
not a given ideal is a \(W\!\)-graph ideal.

In \cite[Section 9]{howvan:wgraphDetSets} it was shown that if
\(J\subseteq K\subseteq S\) and \((\mathscr I_0,J)\) is a \(W_K\)-graph
ideal then \((D_K\mathscr{I}_0,J)\) is a \(W\!\)-graph ideal. This
construction corresponds to inducing modules.
In the present paper we prove a dual result relating to
restriction of modules: if \((\mathscr I\!,\,J)\) is a \(W\!\)-graph ideal 
and \(K \subseteq S\) then for each right coset \(W_Kd\subseteq W\) the
intersection \(\mathscr{I}\cap W_Kd\) is a translate of a 
\(W_K\)-graph ideal. Indeed, for each \(d\in D_K^{-1}\), the set of
minimal right coset representatives for~\(W_K\), the set
\(\mathscr{I}_{d}=W_K \cap \mathscr{I}d^{-1}\) is a \(W_K\)-graph ideal with
respect to \(K\cap dJd^{-1}\). Thus
\begin{equation*}
\mathscr I = \bigsqcup_{d\in D_K^{-1}\cap\mathscr{I}} \mathscr{I}_{d}d,
\end{equation*}
where \((\mathscr{I}_{d},K \cap dJd^{-1})\) is a \(W_K\)-graph ideal in each case.

Finally, as an example, we provide a complete list of \(W\!\)-graph ideals and biideals
for Coxeter groups of type \(I_2(m)\), where \(m \geqslant 2\).

The present paper is organized as follows. In Section~\ref{sec:1}, we provide
basic definitions and facts concerning Coxeter groups and Hecke algebras. In
Section~\ref{sec:2} we review the definition of a \(W\!\)-graph and related
concepts, and in Section~\ref{sec:3} we recall the notion of a \(W\!\)-graph ideal and
the procedure for constructing a \(W\!\)-graph from a \(W\!\)-graph ideal.
In Section~\ref{sec:4} we define \(W\!\)-graph subideals and
show that parabolic induction preserves the strong \(W\!\)-graph subideal
relation, as described above.
In Section~\ref{sec:extra} we define \(W\!\)-graph biideals and show that they
do indeed produce \((W\times W^{\mathrm o})\)-graphs. Section~\ref{sec:5} deals mainly with
the computational characterization of \(W\!\)-graph ideals.
In Section~\ref{sec:6} we prove the decomposition formula mentioned above: if \(\mathscr{I}\) is a
\(W\!\)-graph ideal then the intersection of \(\mathscr{I}\) with any right
coset of any standard parabolic subgroup \(W_K\) is a translate of a
\(W_K\)-graph ideal. The paper ends with Section~\ref{sec:ext}, in which 
\(W\!\)-graph ideals and biideals are investigated for Coxeter groups
of rank~\(2\).

\section{Coxeter groups and Hecke algebras}
\label{sec:1}
Let \((W,S)\) be a Coxeter system and \(l\) the length function
on \(W\) determined by~\(S\). The Bruhat order, denoted by \(\leqslant\),
is the partial order on \(W\) such that \(1\) (the identity element) is
the unique minimal element and the following property holds. 

\begin{lemma}\textup{\cite[Theorem 1.1]{deo:bruhat}}\label{lifting1}
\ Let \(s \in S\) and \(u,\,w\in W\) satisfy \(u\leqslant su\) and \(w\leqslant sw\).
Then \(u\leqslant w\) if and only if \(u\leqslant sw\), and \(u\leqslant sw\)
if and only if \(su\leqslant sw\).
\end{lemma}
The following result follows easily from Lemma~\ref{lifting1}
\begin{lemma}\label{cancellation}
Let \(u,\,v,\,w\in W\) with \(l(uv)=l(u)+l(v)\) and \(l(uw)=l(u)+l(w)\).
Then \(uv\leqslant uw\) if and only if \(v\leqslant w\).
\end{lemma}
As well as the Bruhat order, we shall make extensive use of the left weak order,
defined by the condition that if \(v,\,w\in W\) then \(v\leqslant\lside w\) if and only if
\(l(w)=l(wv^{-1})+l(v)\). The right weak order is defined similarly, and
satisfies \(v\leqslant\rside w\) if and only if \(v^{-1}\leqslant\lside w^{-1}\).

For each \(J\subseteq S\) let \(W_{J}\) be the (standard parabolic)
subgroup of \(W\) generated by~\(J\), and let \(D_{J}\) the set of
distinguished (or minimal) representatives of the left cosets of
\(W_{J}\) in~\(W\). Thus each \(w \in W\) has a unique factorization
\(w = du\) with \(d \in D_{J}\) and \(u \in W_{J}\), and \(l(du) = l(d) +l(u)\)
holds for all \(d \in D_{J}\) and \(u \in W_{J}\). It is easily seen
that \(D_J\) is an ideal of \((W,\leqslant\lside)\): if \(w\in D_J\) and \(v\in W\)
with \(v\leqslant\lside w\) then \(v\in D_J\).

If \(L \subseteq J\subseteq S\) then we define \(D_L^J = W_J \cap D_L\), the set
of minimal representatives of the left cosets of \(W_L\) in~\(W_J\).

If \(W_J\) is finite then we denote the longest element of \(W_{J}\)
by~\(w_{J}\). If \(W\) is finite then \(D_{J} =
\{\,w \in W \mid w \leqslant\lside d_{J}\,\}\) (\cite[Lemma
2.2.1]{gecpfei:charHecke}), where \(d_{J}\) is
the unique element in \(D_{J} \cap w_{S}W_{J}\).

The map \(W \rightarrow D_{J}\) given by \(w = du \mapsto d\)
preserves the Bruhat order, as the following proposition shows.

\begin{proposition}\textup{\cite[Lemma 3.5]{deo:bruhat}}\label{orderPresr}
\ Let \(w_{1} = d_{1}u_{1}\) and \(w_{2} = d_{2}u_{2}\), where \(d_{1}, d_{2} \in D_{J}\) and
\(w_{1}, w_{2} \in W_{J}\). If \(w_{1} \leqslant w_{2}\) then \(d_{1} \leqslant d_{2}\).
\end{proposition}
The following result will be used frequently later.

\begin{lemma}\textup{\cite[Lemma 2.1 (iii)]{deo:paraKL}}\label{deo1}
\ Let \(J \subseteq S\). For each \(s \in S\) and each \(w \in
D_{J}\), exactly one of the following occurs:
\begin{itemize}[topsep=1 pt]
\item[\textup{(i)}] \(l(sw) < l(w)\) and \(sw \in D_{J}\);
\item[\textup{(ii)}] \(l(sw) > l(w)\) and \(sw \in D_{J}\);
\item[\textup{(iii)}] \(l(sw) > l(w)\) and \(sw \notin D_{J}\), and \(w^{-1}sw \in J\).
\end{itemize}
\end{lemma}
Let \(K \subseteq S\). Applying the anti-automorphism of \(W\) given
by \(w \mapsto w^{-1}\) shows that \(D_K^{-1}\) is the set of minimal
representatives of the right cosets of \(W_K\) in \(W\). It is
well known that each double coset \(W_KwW_J\) contains a unique element
\(d\in D_{K,J}= D_K^{-1} \cap D_J\), and that
\(W_K\cap dW_Jd^{-1}=W_{K\cap dJd^{-1}}\) whenever \(d\in D_{K,J}\). 
It follows that each element of \(W_KdW_J\) has a factorization
\(vdu\) with \(v \in D^K_{K\cap dJd^{-1}}\) and \(u \in W_J\), and
satisfying \(l(vdu) = l(v) + l(d) + l(u)\).
Applying this to elements of \(D_J\) gives the following result.
\begin{lemma}\label{mackey}
Let \(J, K \subseteq S\). Then
\(\displaystyle D_{J} = \bigsqcup_{d \in D_{K,J}} D_{K\cap dJd^{-1}}^K d\).
\end{lemma}

\begin{remark}
Each element \(w\) of \(D_{J}\) has a unique factorization \(vd\) with
\(d \in D_{K,J}\) and \(v \in D_{L}^{K}\), where \(L=K\cap dJd^{-1}\),
satisfying \(l(w) = l(v) + l(d)\).
\end{remark}
As in~\cite{howvan:wgraphDetSets}, if \(X\subseteq W\) we
define \(\Pos(X)=\{\,s\in S\mid l(xs)>l(x)\text{ for all }x\in X\,\}\).
Thus \(\Pos(X)\) is the largest subset \(J\) of \(S\) such that
\(X\subseteq D_{J}\).

Let \(\mathcal{A} = \mathbb{Z}[q, q^{-1}]\), the ring of Laurent
polynomials with integer coefficients in the indeterminate \(q\),
and let \(\mathcal{A}^{+} = \mathbb{Z}[q]\). The Hecke algebra
corresponding to the Coxeter system \((W, S)\) is the associative
\(\mathcal{A}\)-algebra \(\mathcal{H}=\mathcal{H}(W) \)
generated by elements \(\{T_{s} \mid s \in S \}\), subject to
the defining relations
\begin{align*}
T^{2}_{s}        &= 1 + (q -q^{-1})T_{s} \quad \text{for all \(s \in S\)},\\
T_{s}T_{t}T_{s}\cdots &= T_{t}T_{s}T_{t}\cdots    \quad
\text{for all \(s, t \in S\)},
\end{align*}
where in the second of these there are \(m(s,t)\) factors on each
side, \(m(s,t)\) being the order of \(st\) in~\(W\).

It is well known that \(\mathcal{H}\) is \(\mathcal{A}\)-free
with an \(\mathcal{A}\)-basis \(\{\, T_{w}\mid w \in W\,\}\) and
multiplication satisfying
\begin{equation*}
T_{s}T_{w} = \begin{cases} T_{sw} & \text{if \(l(sw) > l(w)\),}\\
                            T_{sw} + (q - q^{-1})T_{w} & \text{if \(l(sw) < l(w)\).}
              \end{cases}
\end{equation*}
for all \(s\in S\) and \(w\in W\).

Let \(a \mapsto \overline{a}\) be the involutory automorphism of
\(\mathcal{A}=\mathbb{Z}[q,q^{-1}]\) defined by \(\overline{q}=q^{-1}\).
This extends to an involutory automorphism of \(\mathcal{H}\) satisfying
\begin{equation*}
\overline{T_{s}} = T_{s}^{-1} = T_{s} - (q - q^{-1}) \quad
\text{for all \(s \in S\)}.
\end{equation*}

If \(J\subseteq S\) then \(\mathcal{H}(W_{J})\), the Hecke algebra
associated with the Coxeter system \((W_{J},J)\), is isomorphic to
the subalgebra of \(\mathcal{H}(W)\) generated by
\(\{\,T_{s} \mid s \in J\,\}\). We shall identify \(\mathcal{H}(W_{J})\)
with this subalgebra.

\section{\textit{W-}graphs}
\label{sec:2}

A \(W\!\)-graph is a triple \((V, \mu, \tau)\) consisting of a set
\(V\!\), a function \(\mu\colon V \times V\to \mathbb{Z}\) and
a function \(\tau\) from \(V\) to the power set of \(S\), subject
to the requirement that the free \(\mathcal{A}\)-module with basis~\(V\)
admits an \(\mathcal{H}\)-module structure satisfying
\begin{equation}\label{wgraphdef}
     T_{s}v = \begin{cases}
              -q^{-1}v \quad &\text{if \(s \in \tau(v)\)}\\
              qv + \sum_{\{u \in V \mid s \in \tau(u)\}}\mu(u,v)u
              \quad &\text{if \(s \notin \tau(v)\)},
     \end{cases}
\end{equation}
for all \(s \in S\) and \(v \in V\!\). The elements of \(V\) are the vertices
of the graph, and if \(v\in V\) then \(\tau(v)\) is the colour of the vertex.
By definition there is a directed edge from a vertex \(v\) to a vertex \(u\) if
and only if \(\mu(u,v) \neq 0\), in which case \(\mu(u,v)\) is the weight of
the edge. We say that the edge is \textit{superfluous\/} if
\(\tau(u) \subseteq \tau(v)\) (since the formulas
in Eq.~(\ref{wgraphdef}) would be unchanged by the deletion of any such edge).

\begin{notation}
If \(\Gamma = (V, \mu, \tau)\) is a \(W\!\)-graph, we denote the
\(\mathcal{H}\)-module \(\mathcal{A}V\) by \(M_\Gamma\). When there is no
ambiguity we write \(\Gamma(V)\) for the \(W\!\)-graph whose vertex set is~\(V\).
\end{notation}
Since \(M_\Gamma\) is \(\mathcal{A}\)-free on \(V\) it admits a unique
\(\mathcal{A}\)-semilinear involution \(\alpha \mapsto \overline{\alpha}\)
such that \(\overline v=v\) for all \(v \in V\). We call this involution the
bar involution on \(M_\Gamma\). It is an easy consequence of Eq.~\eqref{wgraphdef}
that \(\overline{h\alpha}=\overline{h}\overline{\alpha}\) for all
\(h\in \mathcal{H}\) and \(\alpha \in \mathcal{A}V\!\).

Following~\cite{kazlus:coxhecke}, define a preorder \(\leqslant_{\Gamma}\) on \(V\)
as follows: \(u \leqslant_{\Gamma} v\) if there exists a sequence of vertices
\(u = x_{0},x_1, \ldots, x_{m} = v\) such that \(\mu(x_{i-1},x_{i}) \neq 0\)
and \(\tau(x_{i-1}) \nsubseteq \tau(x_{i})\) for all \(i \in [1,m]\). That
is, \(u \leqslant_{\Gamma} v\) if there is a directed path
from \(v\) to \(u\) along non-superfluous edges. Let \(\sim_{\Gamma}\)
be the equivalence relation on \(V\) corresponding to \(\leqslant_{\Gamma}\).
The \(\sim_{\Gamma}\) equivalence classes in \(V\) are called 
the \textit{cells} of \(\Gamma\). For each cell \(\mathcal{C}\) the
corresponding full subgraph of \(\Gamma\) is itself a \(W\!\)-graph,
the \(\mu\) and \(\tau\) functions being the restrictions
of those for \(\Gamma\). The preorder \(\leqslant_{\Gamma}\) on \(V\)
induces a partial order on the cells, as follows:
\(\mathcal{C} \leqslant_{\Gamma} \mathcal{C}'\) if \(u \leqslant_{\Gamma} v\) for
some \(u \in \mathcal{C}\) and \(v \in \mathcal{C}'\).

It follows readily from Eq.~(\ref{wgraphdef}) that a subset of
\(V\) spans a \(\mathcal{H}(W)\)\,-submodule of \(M_{\Gamma}\) if
and only if it is closed, in the sense that for every vertex \(v\) in
the subset, each \(u\in V\) satisfying \(\mu(u,v)\ne 0\) and
\(\tau(u)\nsubseteq\tau(v)\) is also in the subset.
Thus \(U\subseteq V\) is a closed subset of \(V\) if and
only if \(U=\bigcup_{v\in U}\{\,u\in V\mid u\leqslant_{\Gamma(V)}v\,\}\).
Clearly, a subset of \(V\) is closed if and only if it is the union
of cells that form an ideal with respect to the partial ordering
of cells. If \(U\) is a closed subset of \(V\) then the subgraphs
\(\Gamma(U)\) and \(\Gamma(V \setminus U)\) induced by \(U\) and
\(V \setminus U\) are themselves \(W\!\)-graphs, with edge weights
\(\mu(v,w)\)and vertex colours \(\tau(v)\) inherited from \(\Gamma(V)\),
and we have
\(M_{\Gamma(V \setminus U)} \cong M_{\Gamma(V)}/M_{\Gamma(U)}\)
as \(\mathcal{H}(W)\)-modules.

It is trivial to check that if \(\Gamma=(V,\mu,\tau)\) is a \(W\!\)-graph
and \(J\subseteq S\) then the \(\mathcal{H}(W_J)\)-module obtained from
\(M_\Gamma\) by restriction is afforded by a \(W_J\)-graph, namely
\(\Gamma_{J} = (V,\mu,\tau_J)\),
where \(\tau_J\) is defined by \(\tau_J(v)=\tau(v)\cap J\)
for all \(v\in V\). We remark that, by the main theorem of \cite{howyin:indwgraph},
if \(N\) is an \(\mathcal{H}(W_J)\)-module afforded by a \(W_J\)-graph
with vertex set~\(U\), then the induced module
\(\mathcal{H}\otimes_{\mathcal{H}(W_J)}N\) is afforded by a \(W\!\)-graph with
vertex set \(D_J\times U\).

We end this section by recalling the original Kazhdan--Lusztig \(W\!\)-graph
for the regular representation of \(\mathcal{H}(W)\).
For each \(w \in W\), define
\begin{align*}
\mathcal{L}(w) &= \{s \in S \mid l(sw) < l(w)\},\\
\mathcal{R}(w) & = \{s \in S \mid l(ws) < l(w)\},
\end{align*}
the elements of which are called the left descents of \(w\) and the right descents
of \(w\), respectively. Kazhdan and Lusztig give a recursive procedure that defines
polynomials \(P_{y,w}\) whenever \(y,\,w \in W\) and \(y < w\). These polynomials satisfy
\(\deg P_{y,w}\leqslant\frac12(l(w)-l(y)-1)\), and \(\mu_{y,w}\) is defined to be
the leading coefficient of \(P_{y,w}\) if the degree is \(\frac12(l(w)-l(y)-1)\),
or 0 otherwise. Now define \(W\opp\) to be the group opposite to \(W\), writing
\(w\mapsto w\opp\) for the natural antiisomorphism from \(W\) to~\(W\opp\).
Observe that \((W\times W^{\mathrm o},S\sqcup S^{\mathrm o})\) is a Coxeter system.
Kazhdan and Lusztig show that defining \(\mu\) and \(\tau\) by the formulas
\begin{align*}
\mu(y,w)    &= \begin{cases}
                    \mu_{y,w} &\quad \text{if \(y < w\)}\\
                    \mu_{w,y} &\quad \text{if \(w < y\)}
            \end{cases}\\
\tau(w) &= \mathcal{L}(w)\sqcup\mathcal{R}(w)^{\mathrm o}
\end{align*}
makes \(\Gamma(W) = (W,\mu,\tau)\) into a \((W\times W^{\mathrm o})\)-graph.
Thus the module \(M_{\Gamma(W)}\) may be regarded as an
\((\mathcal{H},\mathcal{H})\)-bimodule.

\section{\textit{W-}graph ideals}
\label{sec:3}

Let \((W,S)\) be a Coxeter sytem and \(\mathcal{H}=\mathcal{H}(W)\). Let \(\mathscr I\)
be a nonempty ideal in the poset \((W,\leqslant\lside)\), and note that this implies that
\(\Pos(\mathscr{I})=S\setminus\mathscr{I}= \{\,s\in S\mid s\notin\mathscr{I}\,\}\).
Let \(J\) be a subset of \(\Pos(\mathscr{I})\),  so that \(\mathscr I\subseteq D_J\).
For each \(w \in \mathscr I\) the following subsets of~\(S\) give a partition of \(S\):
\begin{align*}
\SD(\mathscr I\!,\,w)&=\{\,s\in S\mid sw<w\,\},\\
\SA(\mathscr I\!,\,w)&=\{\,s\in S\mid sw>w\text{ and }sw\in\mathscr I\,\},\\
\WD_{J}(\mathscr I\!,\,w)&=\{\,s\in S\mid sw>w\text{ and }sw\notin
D_{J}\,\},\\
\WA_{J}(\mathscr I\!,\,w)&=\{\,s\in S\mid sw> w\text{ and }
sw\in D_{J}\setminus\mathscr I\,\}.
\end{align*}
We call the elements of these sets the
strong ascents, strong descents, weak ascents and weak descents
of~\(w\) relative to \(\mathscr{I}\) and~\(J\). If  \(\mathscr{I}\) and~\(J\)
are clear from the context then we may omit reference to them, and write,
for example, \(\WA(w)\) rather than \(\WA_{J}(\mathscr I\!,\,w)\).
We also define \(\D_{J}(\mathscr I\!,\,w) = \SD(\mathscr I\!,\,w)
\cup \WD_{J}(\mathscr I\!,\,w)\) and \(\A_{J}(\mathscr I\!,\,w) =
\SA(\mathscr I\!,\,w) \cup \WA_{J}(\mathscr I\!,\,w)\), the descents
and ascents of~\(w\) relative to \(\mathscr{I}\) and~\(J\).

\begin{remark}
It follows from Lemma \ref{deo1} that
\begin{align*}
\WA(w) &= \{\,s \in S \mid sw \notin \mathscr I \text{ and } w^{-1}sw \notin J\,\},\\
\WD(w) &= \{\,s \in S \mid sw \notin \mathscr I \text{ and } w^{-1}sw \in J\,\},
\end{align*}
since \(sw \notin \mathscr I\) implies that \(sw>w\), given that
\(\mathscr I\) is an ideal in \((W,\leqslant\lside)\). Clearly all descents of the identity
element are weak descents, and in fact \(\D(1)=\WD(1)=J\).
\end{remark}

\begin{definition}\label{wgphdetelt}
With the above notation, we say that \(\mathscr{I}\) is a
\textit{\(W\!\)-graph ideal with respect to~\(J\)}, or that \((\mathscr{I}\!,\,J)\)
is a \(W\!\)-graph ideal, if the following hypotheses are satisfied.
\begin{itemize}[topsep=1 pt]
\item[\textup{(i)}] There is an \(\mathcal{A}\)-free
\(\mathcal{H}\)-module \(\mathscr{S}=\mathscr{S}(\mathscr{I}\!,\,J)\)
with an \(\mathcal{A}\)-basis \(B=\{\,b_{w}\mid w\in\mathscr I\,\}\) on which the
generators \(T_{s}\) act by \begin{equation}\label{S_0action}
T_{s}b_{w} =
\begin{cases}
  b_{sw}  & \text{if \(s \in \SA(w)\),}\\
  b_{sw} + (q - q^{-1})b_{w} & \text{if \(s \in \SD(w)\),}\\
  -q^{-1}b_{w} & \text{if \(s \in \WD(w)\),}\\
  qb_{w} - \sum\limits_{\substack{y \in \mathscr I\\y < sw}} r^s_{y,w}b_{y} &
  \text{if \(s \in \WA(w)\),}
\end{cases}
\end{equation}
for some polynomials \(r^s_{y,w} \in q\mathcal{A}^{+}\!\).
\item[\textup{(ii)}] The module \(\mathscr{S}\) admits an
\(\mathcal{A}\)-semilinear involution \(\alpha \mapsto
\overline{\alpha}\) satisfying \(\overline{b_1}=b_1\) and
\(\overline{h\alpha} =\overline{h}\overline{\alpha}\) for all
\(h\in \mathcal{H}\) and \(\alpha \in \mathscr{S}\).
\end{itemize}
The basis \(B\) in (i) is called the \textit{standard basis\/} of \(\mathscr{S}\),
and the involution \(\alpha \mapsto \overline{\alpha}\) in (ii) is called the
\textit{bar involution\/} on~\(\mathscr{S}\).
\end{definition}

\begin{remark}\label{Rmk:T_wb_1}
An obvious induction on \(l(w)\) shows that \(b_w=T_wb_1\) for all
\(w\in\mathscr I\).\end{remark}

\begin{remark}\label{Rmk:r^s_{w,w}}
In view of the relation \(T_s(T_s-q)=-q^{-1}(T_s-q)\), it follows from Eq.~\eqref{S_0action}
that \(\{\,b_w\mid s\in\WD(w)\,\}\cup\{\,b_{sw}-qb_w\mid s\in\SA(w)\,\}\) spans
the \((-q^{-1})\)-eigenspace of \(T_s\) in~\(\mathscr{S}\!\).
In the case \(s\in\WA(w)\) we deduce that \(r^s_{y,w}=qr^s_{sy,w}\) whenever
\(s\in\SA(y)\), and that \(r^s_{y,w}=0\) whenever \(s\in\WA(y)\).
In particular, \(r^s_{w,w}=0\).
\end{remark}

\begin{definition}
If \(w\in W\) and \(\mathscr{I}=\{\,u\in W\mid u\leqslant\lside w\,\}\) is a
\(W\!\)-graph ideal with respect to some \(J\subseteq S\) then we
say that \(w\) is a \textit{\(W\!\)-graph determining element} associated
with \(J\).
\end{definition}

\begin{remark}
If \(\mathscr{I}\) is a \(W\!\)-graph ideal generated by a \(W\!\)-graph
determining element then it follows from
\cite[Proposition 7.9]{howvan:wgraphDetSets} that, in the case
\(s \in \WA_{J}(\mathscr{I})\) in Part~(i) of Definition~\ref{wgdetset},
the sum \(\sum_{y\in\mathscr{I}\!,\,y<sw}r^{s}_{y,w}b_{y}\) can be replaced by
the simpler \(\sum_{y\in\mathscr{I}\!,\,y<w}r^{s}_{y,w}b_{y}\).
\end{remark}

Let \((\mathscr I,J)\) be a \(W\!\)-graph ideal and let
\(\mathscr{S}(\mathscr{I}\!,\,J)\) be the corresponding \(\mathcal{H}\)-module,
as given in Definition~\ref{wgphdetelt}. From these data one can construct a
\(W\!\)-graph \(\Gamma=\Gamma(\mathscr I,J)\) with
\(M_\Gamma=\mathscr{S}(\mathscr{I}\!,\,J)\). Specifically, the
following results are proved in~\cite{howvan:wgraphDetSets}.

\begin{lemma}\textup{\cite[Lemma 7.2.]{howvan:wgraphDetSets}}\label{uniCbasis1}
\ The module \(\mathscr{S}(\mathscr{I}\!,\,J)\) in
Definition~\ref{wgphdetelt} has a unique \(\mathcal{A}\)-basis 
\(C = \{\,c_{w} \mid w\in \mathscr I\,\}\) such that for all 
\(w \in\mathscr I\) we have \(\overline{c_{w}} = c_{w}\) and
\begin{equation}\label{qpoly1}
b_{w} = c_{w} + q\sum_{y < w} q_{y,w}c_{y}
\end{equation}
for certain polynomials \(q_{y,w} \in \mathcal{A}^{+}\).
\end{lemma}
Define \(\mu_{y,w}\) to be the constant term of~\(q_{y,w}\). The
polynomials \(q_{y,w}\), where \(y < w\), can be computed
recursively by the following formulas.

\begin{corollary}\textup{\cite[Corollary 7.4]{howvan:wgraphDetSets}}\label{recursion}
\ Suppose that \(w<sw\in \mathscr I\) and \(y<sw\). If \(y=w\) then
\(q_{y,sw}=1\), and if \(y\ne w\) we have the following formulas:
\begin{itemize}[topsep=1 pt]
\item[\textup{(i)}] \(q_{y,sw}=qq_{y,w}\) \ if \(s\in\A(y)\),
\item[\textup{(ii)}]
\(q_{y,sw} = -q^{-1}(q_{y,w}-\mu_{y,w})+q_{sy,w}
+\sum_x\mu_{y,x}q_{x,w}\) \ if \(s\in\SD(y)\),
\item[\textup{(iii)}]
\(q_{y,sw} = -q^{-1}(q_{y,w}-\mu_{y,w})+\sum_x\mu_{y,x}q_{x,w}\) \
if \(s\in\WD(y)\),
\end{itemize}
where \(q_{y,w}\) and \(\mu_{y,w}\) are regarded as \(0\) if
\(y\not<w\), and in \textup{(ii)} and \textup{(iii)} the sums
extend over all \(x\in \mathscr I\) such that \(y<x<w\) and \(s\notin\D(x)\).
\end{corollary}

\begin{corollary}\label{polringqyw}
Suppose that \(y,w \in \mathscr{I}\) with \(y < w\). If \(l(w) - l(y)\) is odd then
\(q_{y,w}\) is a polynomial in \(q^{2}\), while if \(l(w) - l(y)\) is even then
\(\mu_{y,w} = 0\) and \(q^{-1}q_{y,w}\) is a polynomial in~\(q^{2}\).
\end{corollary}
\begin{proof}
This follows from Corollary~\ref{recursion} by a straightforward induction
on \(l(w) - l(y)\).\qed
\end{proof}
Let \(\mu\colon C \times C \rightarrow \mathbb{Z}\) be given by
\begin{equation}\label{mu-symmetric}
\mu(c_{y},c_{w}) =
                     \begin{cases}
                         \mu_{y,w} &\text{if \(y < w\)}\\
                         \mu_{w,y} &\text{if \(w < y\)}\\
                         0 &\text{otherwise},
                     \end{cases}
\end{equation}
and let \(\tau\) from \(C\) to the power set of~\(S\) be given by
\(\tau(c_{w}) = \D(w)\) for all~\(y\in\mathscr{I}\!\).

\begin{theorem}\textup{\cite[Theorem 7.5.]{howvan:wgraphDetSets}}\label{main-wg}
\ The triple \((C, \mu, \tau)\) is a \(W\!\)-graph.
\end{theorem}

\begin{definition}\label{wgbasis}
We call \(C = \{\,c_{w} \mid w \in \mathscr{I}\,\}\) the
\textit{\(W\!\)-graph basis\/} of \(\mathscr{S}(\mathscr{I}\!,\,J)\).
\end{definition}
The generators \(T_{s}\) act on the basis elements \(c_{w}\) as
described in the following theorem.

\begin{theorem}\textup{\cite[Theorem 7.3.]{howvan:wgraphDetSets}}\label{main1}
\ Let \(s\in S\) and \(w \in \mathscr I\). Then
\[
T_sc_{w}=
\begin{cases}
-q^{-1}c_{w}&\text{if \(s\in\D(w)\),}\\
qc_{w}+\sum_{y\in \mathcal R(s,w)}\mu_{y,w}c_{y}
&\text{if \(s\in\WA(w)\),}\\
qc_{w}+c_{sw}+\sum_{y\in \mathcal R(s,w)}\mu_{y,w}c_{y} &\text{if
\(s\in\SA(w)\),}
\end{cases}
\]
where the set \(\mathcal R(s,w)\) consists of all \(y\in\mathscr{I}\) such that
\(y < w\) and \(s\in\D(y)\).
\end{theorem}

\begin{corollary}\textup{\cite[Corollary 3.6.(i)]{nguyen:wgideals2}}\label{leftorderpreorder}
\ Let \(x,y \in \mathscr{I}\). If \(x \leqslant\lside y\) then \(c_{y} \leqslant_{\Gamma(C)} c_{x}\).
\end{corollary}

\begin{remark}\label{basiselminq-1es}
It is an easy consequence of Theorem~\ref{main1} that \(\{\,c_{w} \mid s \in \D(w)\,\}\)
is a basis for the \((-q^{-1})\)-eigenspace of \(T_{s}\) in \(M_{\Gamma}\). In particular,
since Eq.~(\ref{S_0action}) shows that \(b_{w}\) is in this eigenspace when
\(s \in \WD(w)\), it follows from Lemma~\ref{uniCbasis1} that \(q_{y,w}=0\)
whenever there is an \(s \in \WD(w)\) such that \(s \notin \D(y)\).
\end{remark}

\begin{corollary}\label{ywmorethan3}
Let \(y,w \in \mathscr{I}\) with \(y < w\) and \(l(y) < l(w) - 1\). If
\(\mu_{y,w} \neq 0\) then \(\D(w) \subseteq \D(y)\).
\end{corollary}
\begin{proof}
Suppose, for a contradiction, that \(\D(w) \cap \A(y)\ne \emptyset\), and choose
\(s \in \D(w) \cap A(y)\). If \(s \in \SD(w)\) then
the first formula in Corollary~\ref{recursion} gives \(q_{y,w} = qq_{y,sw}\), whence
\(\mu_{y,w} = 0\), since \(\mu_{y,w}\) is the constant term of \(q_{y,w}\). But if
\(s \in \WD(w)\) then \(q_{y,w} = 0\) by Remark~\ref{basiselminq-1es},
so that \(\mu_{y,w} = 0\) in this case also. In either case, the assumption
that \(\mu_{y,w} \neq 0\) is contradicted.\qed
\end{proof}

\section{Strong subideals of a \textit{W-}graph ideal}
\label{sec:4}

As above, let \((W,S)\) be a Coxeter system, and \(\mathcal{H} = \mathcal{H}(W)\).

\begin{definition}\label{wgsubideal}
Suppose that \((\mathscr{I}\!,\,J)\) and \((\mathscr{I}_0,J_0)\) are \(W\!\)-graph ideals.
We say that \(\mathscr{I}\) is a \textit{\(W\!\)-graph subideal of \(\mathscr{I}_0\)}
if \(\mathscr{I}\subseteq\mathscr{I}_0\) and \(J=J_0\).
\end{definition}
The following result is Theorem 4.4 of \cite{nguyen:wgideals2}. See
Remark~\ref{remproof} below for some comments relating to its proof.

\begin{theorem}\label{wgdetset}
Let \((\mathscr{I}_{0}, J)\) be a \(W\!\)-graph ideal, and let
\(C_{0} = \{\,c^{0}_{w} \mid w \in \mathscr{I}_{0}\,\}\) be the \(W\!\)-graph
basis of the module \(\mathscr{S}_0=\mathscr{S}(\mathscr{I}_{0},J)\). Suppose that
\(\mathscr I\subseteq\mathscr{I}_0\) 
and \(\{\,c_w^0\mid w\in\mathscr{I}_{0} \setminus \mathscr{I}\,\}\) is a closed subset
of \(C_{0}\). Then \(\mathscr I\) is a \(W\!\)-graph subideal of \(\mathscr{I}_{0}\).
Moreover, the corresponding \(W\!\)-graph \(\Gamma(\mathscr{I})\) is isomorphic to
the full subgraph of \(\Gamma(\mathscr{I}_0)\) on the vertex set
\(\{\,c_w^0\mid w\in\mathscr I\,\}\subseteq C_0\), with \(\tau\) and \(\mu\) functions
inherited from \(\Gamma(\mathscr{I}_{0})\).
\end{theorem}
In view of Theorem~\ref{wgdetset} we make the following definition.

\begin{definition}\label{strongwgsubideal}
Let \((\mathscr{I}_0,\,J)\) be a \(W\!\)-graph ideal and let
\(C_0 = \{\,c_{w}^0 \mid w \in \mathscr{I}_0\}\) be the \(W\!\)-graph
basis of the module \(\mathscr{S}(\mathscr{I}_0,\,J)\). A
\textit{strong \(W\!\)-graph subideal of \(\mathscr{I}_0\)} is a \(W\!\)-graph
subideal \(\mathscr{I}\) such that \(\{\,c_w^0\mid w\in\mathscr{I}_0\setminus \mathscr{I}\}\)
is a closed subset of~\(C_0\).
\end{definition}

\begin{remark}\label{remproof}
Given the hypotheses of Theorem~\ref{wgdetset}, let \(\Gamma(\mathscr{I}_0)=(C_0,\mu,\tau)\)
be the \(W\!\)-graph obtained from \((\mathscr{I}_0,\,J)\), and let \(\mathscr{S}'\) be the
\(\mathcal{A}\)-submodule of \(\mathscr{S}_0 = M_\Gamma\) spanned by the set
\(C'=\{\,c_w^0\mid w\in\mathscr{I}_{0} \setminus \mathscr{I}\,\}\). The assumption
that \(C'\) is closed ensures, by Corollary~\ref{leftorderpreorder},
that \(\mathscr{I}\) is an ideal of \((W,\leqslant\lside)\). Moreover, \(\mathscr{S}'\) is an
\(\mathcal{H}(W)\)-submodule of~\(\mathscr{S}_0\). Now, defining \(f\) to be the
natural map \(\mathscr{S}_0\to\mathscr{S}_0/\mathscr{S}'\), it is readily checked
that for all \(s\in S\) and \(w\in\mathscr{I}\),
\[
     T_{s}f(c_w^0) = \begin{cases}
              -q^{-1}f(c_w^0) \quad &\text{if \(s \in \tau(w)\)}\\
              qf(c_w^0) + \sum_{\{x \in \mathscr{I} \mid s \in \tau(x)\}}\mu(x,w)f(c_x^0)
              \quad &\text{if \(s \notin \tau(w)\)},
     \end{cases}
\]
since \(f(c_y^0)=0\) whenever \(y\in\mathscr{I}_0\setminus\mathscr{I}\). The proof of
Theorem~\ref{wgdetset} proceeds by showing that if
\(\{\,b^{0}_{w}\mid w\in\mathscr{I}_0\,\}\) is the standard basis of \(\mathscr{S}_0\)
then for all \(w\in\mathscr{I}_0\setminus\mathscr{I}\),
\[
f(b^{0}_{w})\, = \!\!\sum_{y\in\mathscr{I}\!\!,\,y<w}\!\!r_{y,w}f(b^{0}_{y})
\]
for some polynomials \(r_{y,w} \in q\mathcal{A}^{+}\!\), with
\(r_{y,w} = q\) if \(y = sw\) for some \(s\in S\). Then Lemma~\ref{xtn} below,
which extends part of the proof of Theorem~\ref{wgdetset} given in
\cite{nguyen:wgideals2}, shows that
\(\mathscr{I}\) satisfies Definition~\ref{wgphdetelt},
with \(\mathscr{S}(\mathscr{I}\!,\,J) = \mathscr{S}_0/\mathscr{S}'\) and with
\(\{\,f(b^{0}_{w})\mid w\in\mathscr{I}\,\}\) as its standard basis.
The proof of Lemma~\ref{xtn}
also shows that \(\Gamma(\mathscr{I})\) inherits its \(\mu\) and \(\tau\)
functions from~\(\Gamma(\mathscr{I}_0)\).

Lemma~\ref{xtn} is needed in the proof of Theorem~\ref{indwgsubideal}
below.
\end{remark}

\begin{lemma}\label{xtn}
Assume that \((\mathscr{I}_0,J)\) is a \(W\!\)-graph ideal and that
\(\mathscr{I}\subseteq\mathscr{I}_0\) is an ideal of \((W,\leqslant\lside)\).
Let \(B_0=\{\,b^{0}_{w}\mid w\in\mathscr{I}_0\,\}\) be the standard basis of
\(\mathscr{S}_0=\mathscr{S}(\mathscr{I}_0,J)\), and
suppose that there exists an \(\mathcal{A}\)-free \(\mathcal{H}\)-module
\(\mathscr{S}\) and an \(\mathcal{H}\)-module homomorphism
\(f\colon\mathscr{S}_0\to\mathscr{S}\) such that
\begin{itemize}[topsep=1 pt]
\item[\textup{(i)}]
\(\{\,f(b^{0}_{w})\mid w\in\mathscr{I}\,\}\) is an \(\mathcal{A}\)-basis
of \(\mathscr{S}\),
\item[\textup{(ii)}]
the kernel of \(f\) is invariant under the bar involution on~\(\mathscr{S}_0\), and
\item[\textup{(iii)}] for each \(w\in\mathscr{I}_0\setminus\mathscr{I}\) and
\(y\in\mathscr{I}\) there is a polynomial \(r_{y,w} \in q\mathcal{A}^{+}\!\)
such that \(r_{y,w} = q\) if \(y = sw\) for some \(s\in S\), and
\(f(b^{0}_{w}) = \sum_{\{y\in\mathscr{I}\mid y<w\}}r_{y,w}f(b^{0}_{y})\).
\end{itemize}
Then \(\mathscr{I}\) is a strong \(W\!\)-graph subideal of \(\mathscr{I}_0\).
\end{lemma}

\begin{proof}
The first step is to show that \((\mathscr{I}\!,\,J)\) is a \(W\!\)-graph ideal.
We define \(b_w=f(b^{0}_{w})\) for all \(w\in\mathscr{I}\!\), so that
by hypothesis \(B=\{\,b_w\mid w\in\mathscr{I}\,\}\) is an \(\mathcal{A}\)-basis of
\(\mathscr{S}\), and proceed to show that the requirements of Definition~\ref{wgphdetelt}
are satisfied. Hypothesis (ii) above ensures that \(\mathscr{S}\) admits a bar involution
such that \(\overline{f(\alpha)}=f(\overline\alpha)\) for all~\(\alpha\in\mathscr{S}_0\),
and the requirements that \(\overline{b_1}=b_1\) and that
\(\overline{h\alpha} =\overline{h}\overline{\alpha}\) for all \(h\in\mathcal{H}\) and
\(\alpha \in\mathscr{S}\) follow immediately by applying \(f\) to
the corresponding formulas in~\(\mathscr{S}_0\).

Since \((\mathscr{I}_0,J)\) is a \(W\!\)-graph ideal and \(f\) is an
\(\mathcal{H}\)-module homomorphism, it follows from Definition~\ref{wgphdetelt}
that for all \(s\in S\) and \(w\in \mathscr{I}_0\),
\[
T_{s}f(b^{0}_{w}) =
\begin{cases}
  f(b^{0}_{sw})  & \text{if \(s \in \SA(\mathscr{I}_0,w)\),}\\
  f(b^{0}_{sw}) + (q - q^{-1})f(b^{0}_{w}) & \text{if \(s \in \SD(\mathscr{I}_0,w)\),}\\
  -q^{-1}f(b^{0}_{w}) & \text{if \(s \in \WD_J(\mathscr{I}_0,w)\),}\\
  qf(b^{0}_{w}) - \sum_{\{y \in \mathscr I\mid y < sw\}} r^s_{y,w}f(b^{0}_{y}) &
  \text{if \(s \in \WA_J(\mathscr{I}_0,w)\),}
  \end{cases}
\]
for some polynomials \(r^s_{y,w} \in q\mathcal{A}^{+}\!\). Note that since
\(\mathscr{I}\subseteq\mathscr{I}_0\) it follows immediately from the definitions
that if \(w\in\mathscr{I}\) then \(\SD(\mathscr{I}, w)=\SD(\mathscr{I}_{0}, w)\) and
\(\WD_{J}(\mathscr{I},w)=\WD_{J}(\mathscr{I}_{0},w)\), and
\(\SA(\mathscr{I})\subseteq\SA(\mathscr{I}_0)\). Thus if \(s\in S\) and
\(w\in\mathscr{I}\) then
\[
T_{s}b_{w} =
\begin{cases}
  b_{sw}  & \text{if \(s \in \SA(\mathscr{I}\!,\,w)\),}\\
  b_{sw} + (q - q^{-1})b_{w} & \text{if \(s \in \SD(\mathscr{I}\!,\,w)\),}\\
  -q^{-1}b_{w} & \text{if \(s \in \WD_J(\mathscr{I}\!,\,w)\),}\\
  qb_{w} - \sum_{\{y \in \mathscr I\mid y < sw\}} r^s_{y,w}b_{y} &
  \text{if \(s \in \WA_J(\mathscr{I}_0,w)\),}
\end{cases}
\]
and to complete the proof that Eq.~(\ref{S_0action}) holds in all cases it remains
to show that it holds whenever \(s\) is in \(\WA_J(\mathscr{I}\!,\,w)\) and in
\(\SA(\mathscr{I}_0,w)\). In this case we have
\(sw\in\mathscr{I}_0\) and \(sw\notin\mathscr{I}\!\), and in view of hypothesis~(iii)
it follows that
\[
T_sb_w = f(b^{0}_{sw}) = \sum_{{\substack{y\in\mathscr{I}\\y<sw}}}r_{y,sw}b_{y}
=qb_w+ \sum_{{\substack{y\in\mathscr{I}\\y<w}}}r_{y,sw}b_{y}
\]
by Lemma~\ref{lifting1} and the fact that \(r_{w,sw}=q\) (by hypothesis).
So Eq.~(\ref{S_0action}) does indeed hold, with \(r^s_{y,w}=-r_{y,sw}\)
when \(s\in \WA_J(\mathscr{I}\!,\,w) \cap \SA(\mathscr{I}_0,w)\), and
hence \((\mathscr{I}\!,\,J)\) is a \(W\!\)-graph ideal.

Now let \(C_0=\{\,c^{0}_{w}\mid w \in \mathscr{I}_{0}\,\}\) be the \(W\!\)-graph basis
of \(\mathscr{S}_0\) and let \(C=\{\,c_{w}\mid w \in \mathscr{I}\,\}\) be the
\(W\!\)-graph basis of \(\mathscr{S}\). Thus, by Theorem~\ref{uniCbasis1}, for
all \(w\in\mathscr{I}_0\) there exist polynomials \(q^{0}_{y,w}\in\mathcal{A}^+\)
such that
\begin{equation}\label{bcI_0}
c^{0}_{w}= b^{0}_{w} - q\sum_{\substack{y < w\\ y\in\mathscr{I}_0}} q^{0}_{y,w}c^{0}_{y}
\end{equation}
and for all \(w\in\mathscr{I}\) there exist polynomials \(q_{y,w}\in\mathcal{A}^+\)
such that
\begin{equation}\label{bcI}
c_{w}= b_{w} - q\sum_{\substack{y < w\\ y\in\mathscr{I}}} q_{y,w}c_{y}.
\end{equation}
We use induction on \(l(w)\) to show that for all \(w\in\mathscr{I}_0\),
\[
f(c^{0}_{w})=\begin{cases}
            c_{w}&\text{if \(w\in\mathscr{I}\),}\\
            0 &\text{if \(w\notin\mathscr{I}\).}\end{cases}
\]
In the course of this we shall also show that \(q_{y,w}=q^{0}_{y,w}\) whenever
\(y,\,w\in\mathscr{I}\) with \(y<w\).

In the case \(l(w)=0\) we have \(w=1\) and \(f(c^{0}_{w})=f(b^{0}_{w})=b_w=c_w\),
as required. Now assume that \(w\in\mathscr{I}_0\) and \(l(w)>1\). Applying \(f\) to both sides
of Eq.~(\ref{bcI_0}) gives
\begin{align*}
f(c^{0}_{w})&= f(b^{0}_{w}) - q\sum_{\substack{y < w\\ y\in\mathscr{I}_0}} q^{0}_{y,w}f(c^{0}_{y})\\
&= f(b^{0}_{w}) - q\sum_{\substack{y < w\\ y\in\mathscr{I}}} q^{0}_{y,w}c_{y}
\end{align*}
by the inductive hypothesis. If \(w\in \mathscr{I}\) then \(f(b^{0}_{w})=b_{w}\), and using
Eq.~(\ref{bcI}) we find that
\[
f(c^{0}_{w})-c_{w}=\sum_{\substack{y < w\\ y\in\mathscr{I}}} q(q_{y,w}-q^{0}_{y,w})c_{y}.
\]
But the left hand side is fixed by the bar involution, as are the basis elements \(c_{y}\)
on the right hand side. So the coefficients \(q(q_{y,w}-q^{0}_{y,w})\) must also be fixed.
But since \(q(q_{y,w}-q^{0}_{y,w})\) is a polynomial in \(q\) with zero constant term, and
since \(\overline{q}=q^{-1}\), this forces \(q(q_{y,w}-q^{0}_{y,w})=0\). Hence
\(f(c^{0}_{w})=c_{w}\) and \(q_{y,w}=q^{0}_{y,w}\), as required. On the other hand, if
\(w\notin \mathscr{I}\) then by our hypothesis~(iii),
\[
f(b^{0}_{w}) = \sum_{{\substack{y<w\\y\in\mathscr{I}}}}r_{y,w}b_{y}
\]
where the \(r_{y,w}\) are polynomials in \(q\) with zero constant term, and so
(using Eq.~\ref{bcI})
\[
f(c^{0}_{w})=\sum_{{\substack{y<w\\y\in\mathscr{I}}}}r_{y,w}\Bigl(c_{y}+
q\sum_{\substack{z < y\\ z\in\mathscr{I}}} q_{z,y}c_{z}\Bigr)
- q\sum_{\substack{y < w\\ y\in\mathscr{I}}} q^{0}_{y,w}c_{y}.
\]
Since \(f(c^{0}_{w})\) is fixed by the bar involution, while the right hand side is
a linear combination of the basis elements \(c_y\) in which all the coefficients
are polynomials with zero constant term, it follows that \(f(c^{0}_{w})=0\), as required.

It is now clear that \(C'=\{\,c^{0}_{w}\mid w\in\mathscr{I}_0\setminus\mathscr{I}\,\}\)
spans an \(\mathcal{H}\)-submodule of~\(\mathscr{S}_0\), namely the kernel of~\(f\). Hence
\(C'\) is a closed subset of~\(C_0\), and so \(\mathscr{I}\) is a strong
\(W\!\)-graph subideal of~\(\mathscr{I}_0\).\qed
\end{proof}

\begin{remark}\label{mu-and-tau}
In the situation of Lemma~\ref{xtn}, let \(\Gamma_0=(C_0,\mu_0,\tau_0)\) be the
\(W\!\)-graph obtained from \(\mathscr{I}_0\) and \(\Gamma=(C,\mu,\tau)\) the
\(W\!\)-graph obtained from \(\mathscr{I}\). Recall that if
\(\mu_{y,w}\) denotes the constant term of the polynomial \(q_{y,w}\), then
for all \(y,\,w\in\mathscr{I}\),
\[
\mu(c_{y},c_{w}) =
                     \begin{cases}
                         \mu_{y,w} &\text{if \(y < w\),}\\
                         \mu_{w,y} &\text{if \(w < y\),}\\
                         0 &\text{otherwise}.
                     \end{cases}
\]
The parameters \(\mu_0(c^{0}_y,c^{0}_{w})\), for \(y,\,w\in\mathscr{I}_0\), are
similarly obtained from the polynomials \(q^{0}_{y,w}\). Since we showed in the
proof that \(q^{0}_{y,w}=q_{y,w}\) whenever \(y,\,w\in\mathscr{I}\) with \(y<w\), it
follows that \(\mu(c_y,c_w)=\mu_0(c^{0}_{y},c^{0}_{w})\) whenever \(y,\,w\in\mathscr{I}\).
Furthermore, \(\tau(c_{w})=\tau_0(c^{0}_{w})\) whenever \(w\in\mathscr{I}\),
since by definition \(\tau(c_{w})=D_J(\mathscr{I}\!,\,w)\) and
\(\tau(c^{0}_{w})=D_J(\mathscr{I}_0,w)\), and, as we noted in the proof, these are equal
if \(w\in\mathscr{I}\), since \(\SD(\mathscr{I}\!,\,w) = \SD(\mathscr{I}_{0},w)\) and
\(\WD_{J}(\mathscr{I}\!,\,w) = \WD_{J}(\mathscr{I}_{0},w)\). Thus \(\Gamma\) is
isomorphic to the full (decorated) subgraph of \(\Gamma_0\) on the
vertices \(\{\,c^{0}_{w}\mid w\in\mathscr{I}\,\}\).
\end{remark}

\begin{remark}\label{characterization}
The converse of Lemma~\ref{xtn} is also true: if \((\mathscr{I}_0,J)\) is a \(W\!\)-graph
ideal and \(\mathscr{I}\) is a strong \(W\!\)-graph subideal of \(\mathscr{I}_0\),
then \(\mathscr{S}=\mathscr{S}(\mathscr{I}\!,\,J)\) is an \(\mathcal{A}\)-free
\(\mathcal{H}\)-module, and there is an \(\mathcal{H}\)-module homomorphism
\(f\colon\mathscr{S}(\mathscr{I}_0,J)\to\mathscr{S}\) satisfying conditions (i), (ii) and~(iii) of
Lemma~\ref{xtn}. Indeed, the proof of Theorem~\ref{wgdetset} proceeded by constructing
the required~\(f\), and in the course of this the following properties of~\(f\) were
established:
\begin{itemize}[topsep=2pt]
\item[(i)]
\(f(c^{0}_{w})=c_w\) for all \(w\in\mathscr{I}\) and \(f(c^{0}_{w})=0\) for all
\(w\in\mathscr{I}_0\setminus\mathscr{I}\),
\item[(ii)]\(\,f(b^{0}_{w})=b_w\) for all \(w\in\mathscr{I}\!\), while for all
\(w\in\mathscr{I}_0\setminus\mathscr{I}\) there exist polynomials
\(r_{y,w}\in q\mathcal{A}^+\) with \(r_{y,w}=q\) if \(wy^{-1}\in S\) and
\(f(b^{0}_{w}) = \sum_{\{y\in\mathscr{I}\mid y<w\}}r_{y,w}f(b^{0}_{y})\),
\item[(iii)]
\(f(\overline{\alpha})=\overline{f(\alpha)}\) for all \(\alpha\in\mathscr{S}_0\).
\end{itemize}
\end{remark}

\begin{proposition}\label{closedIntUnion}
If \(\mathscr{I}_0\) is a \(W\!\)-graph ideal and \(\mathscr{I}_1\) and
\(\mathscr{I}_2\) are strong \(W\!\)-graph subideals of \(\mathscr{I}_0\), then
\(\mathscr{I}_1 \cup \mathscr{I}_2\) and \(\mathscr{I}_1 \cap \mathscr{I}_2\) are
strong \(W\!\)-graph subideals of \(\mathscr{I}_0\).
\end{proposition}

\begin{proof} This is clear, since intersections and unions of ideals of \((W,\leqslant\lside)\) are
ideals, and, for any \(W\!\)-graph, intersections and unions of closed sets are closed. \qed
\end{proof}
We now come to the main result of this section: induction of \(W\!\)-graph ideals preserves
the strong subideal relationship.

\begin{theorem}\label{indwgsubideal}
Suppose that \(J\subseteq K\subseteq S\) and that \((\mathscr{I}_{0},J)\) is a \(W_K\)-graph
ideal. If \(\mathscr{I}\) is a strong \(W_K\)-graph subideal of \(\mathscr{I}_{0}\) then
\(D_{K}\mathscr{I}\) is a strong \(W\!\)-graph subideal of \(D_{K}\mathscr{I}_{0}\).
\end{theorem}

\begin{proof}
Write \(\mathcal{H}_K\) for the Hecke algebra associated with the Coxeter system
\((W_{K},K)\), regarded as a subalgebra of~\(\mathcal{H}\). Let \(\mathscr{S}_{0}\)
and \(\mathscr{S}\) be the \(\mathcal{H}_{K}\)-modules derived fron the \(W_K\)-graphs
\((\mathscr{I}_{0},J)\) and~\((\mathscr{I}\!,\,J)\), and let
\(B_{0} = \{\,b^{0}_{w}\mid w\in \mathscr{I}_{0}\,\}\) and
\(B = \{\,b_{w}\mid w\in \mathscr{I}\,\}\) be their standard bases. By
Remark~\ref{characterization} there is an \(\mathcal{H}_{K}\)-module homomorphism
\(f\colon\mathscr{S}_0\to\mathscr{S}\) satisfying
\begin{itemize}
\item[(i)]
\(f(\overline{\alpha})=\overline{f(\alpha)}\) for all \(\alpha\in\mathscr{S}_0\),
\item[(ii)]\(\,f(b^{0}_{w})=b_w\) for all \(w\in\mathscr{I}\!\), and for all
\(w\in\mathscr{I}_0\setminus\mathscr{I}\) there exist \(r_{y,w}\in q\mathcal{A}^+\)
with \(r_{y,w}=q\) if \(wy^{-1}\in S\) and
\(f(b^{0}_{w}) = \sum_{\{y\in\mathscr{I}\mid y<w\}}r_{y,w}b_{y}\).
\end{itemize}
We know from Theorem~9.2 of \cite{howvan:wgraphDetSets} that \(D_{K}\mathscr{I}_{0}\)
and \(D_{K}\mathscr{I}\) are \(W\!\)-graph ideals, and the associated \(\mathcal{H}\)-modules
are the induced modules
\(\mathscr{S}^{*}_{0}=\mathcal{H}\otimes_{\mathcal{H}_K}\mathscr{S}_{0}\)
and \(\mathscr{S}^{*}=\mathcal{H}\otimes_{\mathcal{H}_K}\mathscr{S}\!\).
Moreover, \(B^{*}_{0}=\{\,T_d\otimes b^{0}_{w}\mid d\in D_K, w\in\mathscr{I}_0\,\}\)
and \(B^{*}=\{\,T_d\otimes b_{w}\mid d\in D_K, w\in\mathscr{I}\,\}\) are the
standard bases of~\(\mathscr{S}^{*}_{0}\) and \(\mathscr{S}^{*}\!\!\), and the bar
involutions satisfy
\(\overline{h\otimes \alpha}=\overline{h}\otimes\overline{\alpha}\) for all
\(h\in\mathcal{H}\) and \(\alpha\) in \(\mathscr{S}_0\) or \(\mathscr{S}\).
Let \(f^{*}\colon\mathscr{S}^{*}_{0}\to\mathscr{S}^{*}\) be the \(\mathcal{H}\)-module
homomorphism induced from the \(\mathcal{H}_K\)-module homomorphism \(f\), so
that \(f^{*}(h\otimes \alpha)=h\otimes f(\alpha)\) for all \(h\in\mathcal{H}\)
and \(\alpha\in\mathscr{S}_{0}\). The conclusion that \(D_{K}\mathscr{I}\)
is a strong \(W\!\)-graph subideal of \(D_{K}\mathscr{I}_{0}\) will follow
by an application of Lemma~\ref{xtn}, if it can be shown that \(f^{*}\) satisfies
conditions (i), (ii) and~(iii) of Lemma~\ref{xtn}.

For all \(d\in D_K\) and \(w\in\mathscr{I}\) we have
\(f^*(T_d\otimes b^{0}_w)=T_d\otimes f(b^{0}_w) =T_d\otimes b_w\), and since
\(\{\,T_d\otimes b_{w}\mid d\in D_K, w\in\mathscr{I}\,\}\) is an \(\mathcal{A}\)-basis
of \(\mathscr{S}^*\), condition~(i) of Lemma~\ref{xtn} is satisfied.

For all \(h\in\mathcal{H}\) and \(\alpha\in\mathscr{S}\) we have
\[
f^{*}(\overline{h\otimes \alpha})=f^{*}(\overline{h}\otimes\overline{\alpha})
=\overline{h}\otimes f(\overline{\alpha})
=\overline{h}\otimes \overline{f(\alpha)}
=\overline{h\otimes f(\alpha)}=\overline{f^{*}(h\otimes \alpha)},
\]
whence \(f^{*}(\overline{\beta})=\overline{f^{*}(\beta)}\) for all \(\beta\in\mathscr{S}^{*}_{0}\),
and condition~(ii) of Lemma~\ref{xtn} is satisfied.

For all \(d\in D_K\) and \(w\in\mathscr{I}_{0}\setminus\mathscr{I}\) we have
\[
f^*(T_d\otimes b^{0}_w)=T_d\otimes f(b^{0}_w) =T_d\otimes \bigl(\sum\nolimits_yr_{y,w}b_y\bigr)
=\sum\nolimits_y r_{y,w}(T_d\otimes b_y)=\sum\nolimits_y r_{y,w}f^{*}(T_d\otimes b^{0}_y),
\]
where the sums extend over all \(y\in\mathscr{I}\) such that \(y<w\). Since
\(r_{y,w}\in q\mathcal{A}^+\) and \(r_{y,w}=q\) if \(wy^{-1}\in S\),
condition~(iii) of Lemma~\ref{xtn} is satisfied.\qed
\end{proof}
Let \((\mathscr{I}\!,\,J)\) be a \(W\!\)-graph ideal and
\(C=\{\,c_w\mid w\in\mathscr{I}\,\}\) the \(W\!\)-graph basis of
\(\Gamma=\Gamma(\mathscr{I}\!,\,J)\). To simplify our terminology, we shall use the
preorder \(\leqslant_\Gamma\) on \(C\) to define a preorder on~\(\mathscr{I}\!\), writing
\(x\leqslant_{\mathscr{I}} y\) if and only if \(c_x\leqslant_\Gamma c_y\), whenever \(x,\,y\in\mathscr{I}\).
In the same spirit, if \(X\subseteq\mathscr{I}\) then we shall say that \(X\)
is \((\mathscr{I}\!,\,J)\)-closed if \(\{\,c_x\mid x\in X\,\}\) is a closed subset
of~\(C\), and we shall call \(X\) a cell of \((\mathscr{I}\!,\,J)\) if
\(\{\,c_x\mid x\in X\,\}\) is a cell of~\(\Gamma\).

\begin{proposition}\label{cellsandstrongsubideals}
Suppose that \((\mathscr{I}\!,\,J)\) is a \(W\!\)-graph ideal and that \(X\) is a
cell of \((\mathscr{I}\!,\,J)\). Let
\(o(X)=\{\,y\in\mathscr{I}\mid x\leqslant_{\mathscr{I}}y\text{ for some }x\in X\,\}\),
the union of the cells \(Y\) of \((\mathscr{I}\!,\,J)\) with
\(X\leqslant_{\mathscr{I}}Y\). Then \(o(X)\) is a strong \(W\!\)-graph
subideal of \((\mathscr{I}\!,\,J)\). Moreover, if \(\mathscr{Z}\subseteq\mathscr{I}\)
then \(\mathscr{Z}\) is a strong \(W\!\)-graph subideal of \((\mathscr{I}\!,\,J)\) if
and only if it is a union of subideals of the above form.
\end{proposition}

\begin{proof}
Let \(\Gamma\) be the  \(W\!\)-graph \(\Gamma(\mathscr{I}\!,\,J)\).
If \(w\in\mathscr{I}\) and \(s\in\SA(w)\) then \(sw\leqslant_{\mathscr{I}}w\), since
\(\mu(sw,w)=1\) (by Theorem~\ref{main1}) and \(D(sw)\nsubseteq D(w)\).
It follows by an induction on \(l(v)-l(w)\) that if \(w,\,v \in\mathscr{I}\)
with \(w\leqslant\lside v\) then \(v\leqslant_{\mathscr{I}}w\). Hence \(o(X)\) is an
ideal of \((W,\leqslant\lside)\). Now suppose that \(z\in\mathscr{I}\setminus o(X)\)
and \(y\leqslant_{\mathscr{I}}z\). Since \(z\in\mathscr{I}\setminus o(X)\) there
is no \(x\in X\) with \(x\leqslant_{\mathscr{I}}z\), and by transitivity of \(\leqslant_{\mathscr{I}}\)
there is no \(x\in X\) with \(x\leqslant_{\mathscr{I}}y\). So
\(y\in\mathscr{I}\setminus o(X)\). Hence \(\mathscr{I}\setminus o(X)\) is
\((\mathscr{I}\!,\,J)\)-closed, and, by Theorem~\ref{wgdetset}, \(o(X)\) is a
strong \(W\!\)-graph subideal of \((\mathscr{I}\!,\,J)\).

As noted in Proposition~\ref{closedIntUnion}, any union of strong \(W\!\)-graph
subideals is a strong \(W\!\)-graph subideal. Now let \(\mathscr{Z}\) be an arbitrary strong
\(W\!\)-graph subideal of \(\mathscr{I}\!\), and suppose that \(X\) and \(Y\) are cells
of \((\mathscr{I}\!,\,J)\) with \(X\leqslant_{\mathscr{I}}Y\). Since \(\mathscr{I}\setminus\mathscr{Z}\)
is a closed set, if \(Y\subseteq(\mathscr{I}\setminus\mathscr{Z})\) then
\(X\subseteq(\mathscr{I}\setminus\mathscr{Z})\). Equivalently, if \(X\subseteq\mathscr{Z}\)
then \(Y\subseteq\mathscr{Z}\). So if \(X\subseteq\mathscr{Z}\) is a cell
then \(o(X)\subseteq\mathscr{Z}\), and it follows that \(\mathscr{Z}\) is
the union of those strong subideals \(o(X)\) that it contains.\qed
\end{proof}
Combining Theorem~\ref{indwgsubideal} and Proposition~\ref{cellsandstrongsubideals}
yields the following corollary.
\begin{corollary}\label{indcells}
Suppose that \(J\subseteq K\subseteq S\) and that \((\mathscr{I}\!,\,J)\) is a \(W_K\)-graph
ideal. If \(X\subseteq\mathscr{I}\) is a cell of \((\mathscr{I},J)\) then \(D_KX\) is a
union of cells of the induced \(W\!\)-graph ideal \((D_K\mathscr{I},J)\).
\end{corollary}

\begin{proof}
By Proposition~\ref{cellsandstrongsubideals}, the sets
\(o(X)=\{\,y\in\mathscr{I}\mid x\leqslant_{\mathscr{I}}y\text{ for some }x\in X\,\}\)
and \(o(X)\setminus{X}\) are both strong \(W_K\)-graph subideals of \(\mathscr{I}\!\).
So by Theorem~\ref{indwgsubideal} it
follows that \(D_Ko(X)\) and \(D_K(o(X)\setminus X)\) are strong
\(W\!\)-graph subideals of \((D_K\mathscr{I},J)\), and hence their complements
in \(D_K\mathscr{I}\) are unions of cells. Since
\(D_KX=D_Ko(X)\setminus D_K(o(X)\setminus X)\) we deduce that
\(D_KX\) is a union of cells.\qed
\end{proof}

\begin{remark}\label{recoverequalparameter}
Applying Corollary~\ref{indcells} in the case \((\mathscr{I}\!,\,J)=(W_K,\emptyset)\)
recovers the equal parameters case of \cite[Theorem 1]{geck:onIndcells}.
\end{remark}
Let \(\Gamma=(C,\mu,\tau)\) be the \(W\!\)-graph obtained from \(W\!\)-graph ideal
\((\mathscr{I}\!,\,J)=(W,\emptyset)\), so that \(\mathscr{S}(\mathscr{I}\!,\,J)\)
can be identified with the left regular \(\mathcal{H}\)-module, the basis
\(C=\{\,c_w\mid w\in W\,\}\) is the Kazhdan--Lusztig basis of~\(\mathcal{H}\),
and \(\tau(c_w)=\mathcal{L}(w)=\{\,s\in S\mid sw<w\,\}\), for all \(w\in W\).
Observe that every edge of \(\Gamma\) with tail \(c_1\) is superfluous, since
\(\mathcal{L}(1)=\emptyset\subseteq\mathcal{L}(w)\) for all \(w\in W\). Hence
\(W\setminus\{1\}\) is a closed set of \((W,\emptyset)\), and, since \(\{1\}\) is
an ideal of \((W,\leqslant\lside)\), it follows that \(\{1\}\) is a strong
\(W\!\)-graph subideal of~\(W\). Similarly, if \(W\) is finite and \(w_S\) is the
longest element of~\(W\), then every edge of \(\Gamma\) with head \(c_{w_S}\) is
superfluous, since \(\mathcal{L}(w)\subseteq S=\mathcal{L}(w_S)\) for all \(w\in W\).
So \(\{w_S\}\) is \((W,\emptyset)\)-closed. Since \(W\setminus\{w_S\}\) is an
ideal of \((W,\leqslant\lside)\), it follows that \(W\setminus\{w_S\}\) is strong
\(W\!\)-graph subideal of~\(W\).

Since \(\{w_S\}\) is \((W,\emptyset)\)-closed, \(\mathcal{A}c_{w_S}\) is an
\(\mathcal{H}\)-submodule of~\(\mathcal{H}\), as was already obvious from the fact
that \(T_sc_{w_S}=-q^{-1}c_{w_S}\) for all \(s\in S\) (by Theorem~\ref{main1}). Using
this it is also easy to show that \(c_{w_S}=\sum_{w\in W}(-q)^{l(w_s)-l(w)}T_w\).

Now let \(K\subseteq S\). By the above discussion, \(\{1\}\) is a strong
\(W_K\)-graph subideal of \((W_K,\emptyset)\), and so by Theorem~\ref{indwgsubideal}
it follows that \(D_K\) is a strong \(W\!\)-graph subideal of
\((D_KW_K,\emptyset)=(W,\emptyset)\). Thus \(W\setminus D_K\) is a closed
subset of \((W,\emptyset)\), whence \(W\setminus D_K\) and \(D_K\) are both unions
of left cells. Furthermore, if \(W_K\) is finite and \(w_K\) is its longest element,
then \(W_K\setminus\{w_K\}\) is a strong \(W_K\)-graph ideal of \((W_K,\emptyset)\),
and by Theorem~\ref{indwgsubideal} it follows that \(W\setminus D_Kw_K\) is a
strong \(W\!\)-graph subideal of \((W,\emptyset)\). Hence \(D_Kw_K\) is
\((W,\emptyset)\)-closed, and, in particular, \(D_Kw_K\) is a union of left
cells. (This result was proved by Geck in \cite[Lemma~2.8]{geck:klMurphy}.)

It is easily checked, using Definition~\ref{wgphdetelt}, that if \(K\) is any
subset of~\(S\) then \(({1},K)\) is a \(W_K\)-graph ideal, associated with the
one-dimensional representation \(\varepsilon\) of \(\mathcal{H}_K\) given by
\(\varepsilon(T_s)=-q^{-1}\) for all \(s\in K\). By Theorem~\ref{wgdetset} it
follows that \((D_K,K)\) is a \(W\!\)-graph ideal, associated with the
representation of \(\mathcal{H}\) induced from~\(\varepsilon\). (This corresponds
to the case \(u=-1\) in the construction given by Deodhar in~\cite{deo:paraKL}.)
In the case that \(W_K\) is finite with \(w_K\) its longest element, the
\((W_K,\emptyset)\)-closed set \(\{w_K\}\) also affords the representation
\(\varepsilon\), and the \((W,\emptyset)\)-closed set \(D_Kw_K\) also affords the
representation of \(\mathcal{H}\) induced from~\(\varepsilon\). The following
proposition confirms that the \(W\!\)-graph \(\Gamma(D_K,K)\) is isomorphic to the full
subgraph of \(\Gamma(W,\emptyset)\) spanned by the vertices corresponding to~\(D_Kw_K\).

\begin{proposition}
Let \(K\subseteq S\) with \(W_K\) finite. Let \(C=\{\,c_w\mid w\in W\,\}\)
be the \(W\!\)-graph basis of \(\mathscr{S}(W,\emptyset)\) and \(\Gamma=(C,\mu,\tau)\)
the corresponding \(W\!\)-graph, and let \(C^K=\{\,c_d^K\mid d\in D_K\,\}\) be the
\(W\!\)-graph basis of \(\mathscr{S}(D_K,K)\) and \(\Gamma^K=(C^K,\mu^K,\tau^K)\)
the corresponding \(W\!\)-graph. Define \(\varphi\colon C^K\to C\) by
\(\varphi(c_d^K)=c_{dw_K}\) for all~\(d\in D_K\), where \(w_K\) is the longest
element of~\(W_K\). Then \(\tau^K(v)=\tau(\varphi(v))\) for
all \(v\in C^K\), and \(\mu^K(u,v)=\mu(\varphi(u),\varphi(v))\) for all \(u,\,v\in C^K\).
\end{proposition}
\begin{proof}
As above, we identify \(\mathscr{S}(W,\emptyset)\) with~\(\mathcal{H}\).
Since the set \(D_Kw_K\) is \((W,\emptyset)\)-closed, the \(\mathcal{A}\)-submodule
of \(\mathcal{H}\) spanned by \(\{\,c_w\mid w\in D_Kw_K\,\}\) is an
\(\mathcal{H}\)-submodule. It clearly coincides with the left ideal
\(\mathcal{H}c_{w_K}=\bigoplus_{d\in D_K}T_d\mathcal{H}_Kc_{w_K}\). Here each
summand has dimension~1.

The module \(\mathscr{S}(D_K,K)\) can be identified with \(\mathcal{H}c_{w_K}\),
with \(\{\,T_dc_{w_K}\mid d\in D_K\,\}\) as the standard basis, since the bar involution
on \(\mathcal{H}\) fixes \(T_1c_{w_K}\), and for all \(s\in S\) and \(d\in D_K\),
\[
T_{s}T_dc_{w_K} =
\begin{cases}
  T_{sd}c_{w_K}  & \text{if \(sd \in D_K\) and \(sd>d\),}\\
  T_{sd}c_{w_K} + (q - q^{-1})T_{d}c_{w_K} & \text{if \(sd<d\),}\\
  -q^{-1}T_{d}c_{w_K} & \text{if \(sd=dt\) for some \(t\in K\),}
  \end{cases}
\]
in accordance with the requirements of Definition~\ref{wgphdetelt}.
The first of the three cases corresponds to \(s\in\SA(D_K,d)\), the second to
\(s\in\SD(D_K,d)\), the third to \(s\in\WD_K(D_K,d)\). It is immediate from
the definition that \(\WA_K(D_K,d)\) is always empty.

Note that if \(d\in D_K\) and \(s\in S\) then \(sdw_K<dw_K\) if and only if either \(sd<d\)
or \(sd=dt\) for some \(t\in K\). Since \(\tau^K(c_d)=\SD(D_K,d)\cup\WD_K(D_K,d)\)
and \(\tau(\varphi(c_d))=\tau(c_{dw_K})=\mathcal{L}(dw_K)\), this establishes the first
assertion of the proposition.

It follows from Lemma~\ref{uniCbasis1} that the \(W\!\)-graph basis and standard basis
of \(\mathscr{S}(D_K,K)\) are related by
\begin{equation}\label{stdbas}
c^K_d=T_dc_{w_K}-q\sum_{e<d}p_{e,d}^KT_ec_{w_K}\qquad\text{for all \(d\in D_K\),}
\end{equation}
for some \(p_{e,d}^K\in\mathcal{A}^+\!\!\). Moreover, the \(W\!\)-graph basis is the
only basis of bar-invariant elements satisfying such a system of equations. Similarly,
in \(\mathscr{S}(W,\emptyset)\) we have
\[
c_w=T_w-q\sum_{y<w}p_{y,w}T_v\qquad\text{for all \(w\in D_K\),}
\]
for some \(p_{y,w}\in\mathcal{A}^+\!\!\). We apply this with \(w=dw_K\), where \(d\in D_K\),
and group the terms on the right hand side according to cosets of~\(W_K\), thus obtaining
the components of \(c_{dw_K}\) in the direct sum decomposition
\(\mathcal{H}=\bigoplus_{e\in D_K}T_e\mathcal{H}_K\). We find that
\begin{equation}\label{components}
c_{dw_K}=T_d\bigl(T_{w_K}-q\sum_{v<w_K}p_{dv,dw_K}T_v\bigr)
-q\!\!\sum_{e\in D_K,\,e<d}\!T_e\bigl(\sum_{v\in W_K}p_{ev,dw_K}T_v\bigr).
\end{equation}
Since \(c_{dw_K}\in\mathcal{H}c_{w_K}\) its component in each summand
\(T_e\mathcal{H}_K\) must lie in the one-dimensional subspace \(T_e\mathcal{H}_Kc_{w_S}\).
So it follows that \(T_{w_K}-q\sum_{v<w_K}p_{dv,dw_K}T_v\) and each
\(\sum_{v\in W_K}p_{ev,dw_K}T_v\) in Eq.~(\ref{components}) must be scalar multiples
of~\(c_{w_K}=T_{w_K}-q\sum_{v<w_K}(-q)^{l(w_s)-l(v)-1}T_v\). So
\[
c_{dw_K}=T_dc_{w_K}-q\sum_{e<d}p_{ew_K,dw_K}T_ec_{w_K}.
\]
Comparing this with Eq.~(\ref{stdbas}), uniqueness tells us that \(c^K_d=c_{dw_K}\)
for all \(d\in D_K\), and that \(p_{e,d}^K=p_{ew_K,dw_K}\) for all \(e,\,d\in D_K\).
Since \(\mu^K(c_e,c_d)\) is
the constant term of~\(p_{e,d}^K\) and \(\mu(c_{ew_K},c_{dw_K})\) is the constant
term of~\(p_{ew_K,dw_K}\), this establishes the other assertion of the proposition.
\qed
\end{proof}

\begin{remark} The equation \(p_{e,d}^K=p_{ew_K,dw_K}\), which is the key part of the
above proof, is due to Deodhar \cite[Proposition 3.4]{deo:paraKL}. The proof also
shows that \(p_{ev,dw_K}=q^{\smash{l(w_K)-l(v)}}p_{ew_K,dw_K}\) whenever
\(e,\,d\in D_K\) and \(v\in W_K\), a fact that was already known.
\end{remark}

\section{\textit{W-}graph biideals}
\label{sec:extra}

It is clear from the defining presentation that the Hecke algebra
\(\mathcal{H}\) possesses an involutive antiautomorphism \(h\mapsto h^\flat\)
that fixes each element of the generating set \(\{\,T_s\mid s\in S\,\}\). 
This can be used to convert left
\(\mathcal{H}\)-modules into right \(\mathcal{H}\)-modules, and vice versa.
The corresponding antiautomorphism of~\(W\), given by \(w\mapsto w^{-1}\),
maps ideals of \((W,\leqslant\lside)\) to ideals of \((W,\leqslant\rside)\), and vice versa.
Since, moreover, 
\(
(\overline{T_s})^\flat = (T_s^{-1})^\flat = (T_s^\flat)^{-1} = T_s^{-1}
= \overline{T_s} = (\overline{T_s^\flat})
\)
for all \(s\in S\), it follows that \(\overline{h^\flat} = (\overline h)^\flat\)
for all \(h\in H\). So
there is a theory of \(W\!\)-graph right ideals that is completely parallel
to the theory of \(W\!\)-graph (left) ideals as presented above, with
\((W,\leqslant\rside)\) replacing \((W,\leqslant\lside)\) and right \(\mathcal{H}\)-modules
replacing left \(\mathcal{H}\)-modules. Just as \(W\!\)-graph ideals give rise
to \(W\!\)-graphs, so \(W\!\)-graph right ideals give rise
to \(W\opp\!\)-graphs. If \(\mathscr{I}\subseteq W\)
and \(K\subseteq S\) then \((\mathscr{I}\!,\,K)\) is a \(W\!\)-graph right
ideal if and only if \((\mathscr{I}^{-1}\!,\,K)\) is a \(W\!\)-graph ideal.

If \((\mathscr{I}\!,\,K)\) is a \(W\!\)-graph right ideal we write
\(\mathscr{S}\opp(\mathscr{I}\!,\,K)\) for the associated right
\(\mathcal{H}\)-module, \({B\vtms}\opp=\{\,b_w\opp\mid w\in\mathscr{I}\,\}\) for its
standard basis and \({C\vtms}\opp=\{\,c_w\opp\mid w\in\mathscr{I}\,\}\) for its
\(W\opp\!\)-graph basis. The module \(\mathscr{S}\opp(\mathscr{I}\!,\,K)\) admits an
\(\mathcal{A}\)-semilinear involution \(\alpha\mapsto\underline{\alpha}\)
such that \(\underline{\alpha h}=\underline{\alpha}\overline{h}\) for all
\(h\in\mathcal{H}\) and \(\alpha\in\mathscr{S}\opp(\mathscr{I}\!,\,K)\)
and \(\underline{c_w\opp}=c_w\opp\) for all \(w\in\mathscr{I}\!\).
Moreover, as in Lemma~\ref{uniCbasis1}, the \(c_w\opp\) are uniquely
determined by the requirements that \(\underline{c_w\opp}=c_w\opp\)
and \(b_w\opp=c_w\opp+q\sum_{y<w}q_{y,w}\opp c_y\opp\) for some
\(q_{y,w}\opp\in\mathcal{A}^+\!\). We write \(\mu_{y,w}\opp\) for the constant
term of the polynomial \(q_{y,w}\opp\).

\begin{remark}\label{right-left}
If \((\mathscr{I}\!,\,K)\) is a \(W\!\)-graph right ideal then the module
\(\mathscr{S}\opp(\mathscr{I}\!,\,K)\) can be identified with
\(\mathscr{S}(\mathscr{I}^{-1}\!,\,K)\), made into a right module by defining 
\(\alpha h = h^\flat\alpha\) for all \(\alpha\in\mathscr{S}(\mathscr{I}^{-1}\!,\,K)\)
and \(h\in\mathcal{H}\). With this convention, \(b_w\opp=b_{w^{-1}}\),
and Eq.~\eqref{S_0action} says
that for all \(w\in\mathscr{I}\) and \(s\in S\),
\begin{equation}\label{S-1Action}
b_{w}\opp T_{s} =
\begin{cases}
  b_{ws}\opp  & \text{if \(s\in\SA(w^{-1},\mathscr{I}^{-1})\),}\\
  b_{ws}\opp + (q - q^{-1})b_{w}\opp & \text{if \(s\in\SD(w^{-1},\mathscr{I}^{-1})\),}\\
  -q^{-1}b_{w}\opp & \text{if \(s\in\WD_K(w^{-1},\mathscr{I}^{-1})\),}\\
  qb_{w}\opp - \sum_{y \in \mathscr I\!\!,\,\, y < ws} r^s_{y^{-1},w^{-1}}b_{y}\opp &
  \text{if \(s\in\WA_K(w^{-1},\mathscr{I}^{-1})\),}
\end{cases}
\end{equation}
where the coefficients \(r^s_{\smash{y^{-1},w^{-1}}}\)
lie in \(q\mathcal{A}^{+}\). Note that the first of these four cases corresponds to 
\(w<ws\in\mathscr{I}\), the second to \(w>ws\), the third to \(ws\notin D_K^{-1}\!\),
and the last to \(ws \in D_K^{-1}\setminus\mathscr{I}\).
\end{remark}

\begin{remark}\label{JandK}
It is conceivably possible for some \(\mathscr{I}\subseteq W\) to be simultaneously
a \(W\!\)-graph ideal with respect to~\(J\) and a \(W\!\)-graph right ideal with
respect to~\(K\), where \(J,\,K\subseteq S\). However, if this happens then
\(\mathscr{I}\) must be contained in the standard parabolic subgroup generated
by the complement of \(J\cup K\) in~\(S\). To see this, observe that since
\(\mathscr{I}\) is both an ideal of \((W,\leqslant\lside)\) and an ideal of \((W,\leqslant\rside)\),
if \(w\in \mathscr{I}\) and \(u\in W\) has the property that there exist
\(x,\,y\in W\) with \(w=xuy\) and \(l(w)=l(x)+l(u)+l(y)\), then \(u\in\mathscr{I}\!\).
In particular, if \(s\in S\) occurs in any reduced expression for any
\(w\in\mathscr{I}\) then \(s\in\mathscr{I}\!\), whence \(s\notin J\cup K\)
(since \(\mathscr{I}\subseteq D_J\cap D_K^{-1}\)). Of course this will automatically
hold if \(J=K=\emptyset\).
\end{remark}
If it is the case that \((\mathscr{I},J)\) is a \(W\!\)-graph ideal and \((\mathscr{I},K)\)
is a \(W\!\)-graph right ideal then there is an \(\mathcal{A}\)-isomorphism
from the left \(\mathcal{H}\)-module \(\mathscr{S}(\mathscr{I}\!,\,J)\) to the
right \(\mathcal{H}\)-module \(\mathscr{S}\opp(\mathscr{I}\!,\,K)\) mapping the standard
basis of \(\mathscr{S}(\mathscr{I}\!,\,J)\) to the standard basis of
\(\mathscr{S}\opp(\mathscr{I}\!,\,K)\). It is therefore natural to ask
whether it is possible to obtain an \((\mathcal{H},\mathcal{H})\)-bimodule by
identifying \(b_w\opp\) with \(b_w\) for all~\(w\in\mathscr{I}\).
Accordingly, we make the following definition.

\begin{definition}\label{wgbiideal}
Let \(\mathscr{I}\subseteq W\) and \(J,\,K\subseteq S\), and suppose that
\((\mathscr{I}\!,\,J)\) is a \(W\!\)-graph ideal and \((\mathscr{I}\!,\,K)\) is a
\(W\!\)-graph right ideal. Identify \(\mathscr{S}\opp(\mathscr{I}\!,\,K)\) with
\(\mathscr{S}(\mathscr{I}\!,\,J)\) by putting \(b_w\opp = b_w\) for
all~\(w\in\mathscr{I}\).
We say that \(\mathscr{I}\) is a \textit{\(W\!\)-graph biideal
with respect to \(J\) and \(K\)} (or that \((\mathscr{I}\!,\,J,K)\) is a \(W\!\)-graph biideal)
if \(\mathscr{S}=\mathscr{S}(\mathscr{I}\!,\,J)=\mathscr{S}\opp(\mathscr{I}\!,\,K)\) is an
\((\mathcal{H},\mathcal{H})\)-bimodule with the left and right
\(\mathcal{H}\)-actions defined in Eq.~\eqref{S_0action} and Eq.~\eqref{S-1Action}.
\end{definition}

\begin{notation}
When \((\mathscr{I}\!,\,J,K)\) is a \(W\!\)-graph biideal the
\((\mathcal{H},\mathcal{H})\)-bimodule
\(\mathscr{S}(\mathscr{I}\!,\,J)=\mathscr{S}\opp(\mathscr{I}\!,\,K)\)
will be denoted by \(\mathscr{S}(\mathscr{I}\!,\,J,K)\).
\end{notation}
Suppose now that \((\mathscr{I}\!,\,J)\) is simultaneously a \(W\!\)-graph ideal and
a \(W\!\)-graph right ideal, and that
\(\mathscr{S}=\mathscr{S}(\mathscr{I}\!,\,J)=\mathscr{S}\opp(\mathscr{I}\!,\,J)\)
with \(b_w\opp=b_w\) for all \(w\in\mathscr{I}\!\). By Remark~\ref{Rmk:T_wb_1} and
its analogue for the right action, we see that \(T_wb_1 = b_w = b_1T_w\) for
all \(w\in\mathscr I\). The following result shows that \((\mathscr{I}\!,\,J,J)\)
is a \(W\!\)-graph biideal if and only if \(T_wb_{1} = b_{1}T_w\) for all \(w\in W\). 

\begin{lemma}\label{symmetryBasis}
With the assumptions of the above preamble, \(\mathscr{S}\) is an
\((\mathcal{H},\mathcal{H})\)-bimodule if and only if \(hb_{1} = b_{1}h\) for all
\(h \in \mathcal{H}\).
\end{lemma}
\begin{proof}
Suppose first that \(hb_{1} = b_{1}h\) for all \(h \in \mathcal{H}\). Then for all
\(h,\,g\in \mathcal{H}\), we have
\begin{equation}\label{bimoduleeq1}
(hb_{1})g = (b_{1}h)g = b_{1}(hg) = (hg)b_{1} = h(gb_{1}) = h(b_{1}g).
\end{equation}
Now let \(w\) be an arbitrary element of \(\mathscr{I}\). By Remark~\ref{Rmk:T_wb_1}
we have \(b_w=T_wb_1\), and so it follows from Eq.~(\ref{bimoduleeq1}) that for all
\(h,\,g\in \mathcal{H}\),
\[
h(b_wg)=h((T_wb_1)g)=h(T_w(b_1g))=(hT_w)(b_1g)=((hT_w)b_1)g=(h(T_wb_1))g=(hb_w)g.
\]
Since \(\{\,b_w\mid w\in\mathscr{I}\,\}\) spans \(\mathscr{S}\) it follows from this that
\(h(\alpha g)=(h\alpha)g\) for all \(h,\,g\in \mathcal{H}\) and \(\alpha\in\mathscr{S}\),
whence \(\mathscr{S}\) is an \((\mathcal{H},\mathcal{H})\)-bimodule, as required.

Conversely, suppose that \(\mathscr{S}\) is a \((\mathcal{H},\mathcal{H})\)-bimodule.
We must show that \(hb_1=b_1h\) for all \(h \in \mathcal{H}\), and since 
\(\{\,T_w\mid w\in W\,\}\) spans \(\mathcal{H}\) it suffices to show that
\(T_{w}b_{1} = b_{1}T_{w}\) for all \(w\in W\). We use induction on \(l(w)\)
to do this. The case \(l(w)=0\) is trivial. For the inductive step, suppose that
\(l(w) > 0\) and write \(w = sv\) with \(s \in S\) and \(l(v) = l(w) - 1\).
By Eq.~\eqref{S_0action} we find that
\begin{equation}\label{Tsonb1}
T_{s}b_{1} = \begin{cases}
               b_{s}        \quad &\text{if \(s \in \mathscr{I}\),}\\
               -q^{-1}b_{1} \quad &\text{if \(s \notin D_J\),}\\
               qb_{1}       \quad &\text{if \(s \in D_J\setminus\mathscr{I}\),}
             \end{cases}
\end{equation}
and by Eq.~\eqref{S-1Action} it follows that \(b_{1}T_{s}=T_{s}b_{1}\) (since
\(s\notin D_J^{-1}\) if and only if \(s\notin D_J\), as \(s=s^{-1}\)). Hence, by the
inductive hypothesis and the assumption that \(\mathscr{S}\) is a bimodule, it follows that
\begin{equation*}
T_{w}b_{1} = (T_{s}T_{v})b_{1} = T_{s}(T_{v}b_{1}) = T_{s}(b_{1}T_{v})
=(T_{s}b_{1})T_{v} = (b_{1}T_{s})T_{v} = b_{1}(T_{s}T_{v}) = b_{1}T_{w}
\end{equation*}
as required.
\qed
\end{proof}
If \((\mathscr{I}\!,\,J,K)\) is a \(W\!\)-graph biideal then the bimodule
\(\mathscr{S}(\mathscr{I}\!,\,J,K)=\mathscr{S}(\mathscr{I}\!,\,J)=\mathscr{S}\opp(\mathscr{I}\!,\,K)\)
possesses a \(W\!\)-graph basis \(C=\{\,c_w\mid w\in\mathscr{I}\,\}\) and a \(W\opp\!\)-graph
basis \({C\,}\opp=\{\,c_w\opp\mid w\in\mathscr{I}\,\}\). By Lemma~\ref{uniCbasis1} the
\(c_w\) are characterized by the properties that \(\overline{c_{w}} = c_{w}\) and
\(b_{w} = c_{w} + q\sum_{y < w} q_{y,w}c_{y}\)
for some \(q_{y,w} \in \mathcal{A}^{+}\), and similarly the \(c_w\opp\)
are characterized by the properties that \(\underline{c_{w}\opp} = c_{w}\opp\) and
\(b_{w} = c_{w}\opp + q\sum_{y < w} q_{y,w}\opp c_{y}\opp\) for some
\(q_{y,w}\opp \in \mathcal{A}^{+}\). It follows that if
\(\underline{\alpha}=\overline{\alpha}\) for all \(\alpha\in\mathscr{S}(\mathscr{I}\!,\,J,K)\)
then  the \(W\!\)-graph basis \(C\) and the \(W\opp\!\)-graph basis \({C\vtms}\opp\) coincide.

\begin{proposition}\label{involutionsequal}
If \((\mathscr{I}\!,\,J,K)\) is a \(W\!\)-graph biideal then
\(\underline{\alpha}=\overline{\alpha}\) for all \(\alpha\in\mathscr{S}(\mathscr{I}\!,\,J,K)\).
\end{proposition}
\begin{proof}
We use induction on \(l(w)\) to show that \(\underline{b_w}=\overline{b_w}\) for
all \(w\in\mathscr{I}\!\). Since the case \(l(w)=0\) is trivial, assume that \(l(w)>0\)
and let \(w=sv\) with \(s\in S\) and \(l(v)=l(w)-1\). Note that since \(\mathscr{I}\) is
an ideal of \((W,\leqslant\lside)\) and of \((W,\leqslant\rside)\), both \(v\) and \(s\) are
elements of~\(\mathscr{I}\!\). Observe that
\[
\overline{T_s}b_1=(T_s-(q-q^{-1}))b_1=b_s-(q-q^{-1})b_1=b_1(T_s-(q-q^{-1}))=b_1\overline{T_s}.
\]
Hence, by the inductive hypothesis and the fact
that \(\mathscr{S}(\mathscr{I}\!,\,J,K)\) is a bimodule, we find that
\begin{multline*}
\qquad
\underline{b_w}=\underline{b_1T_w}=b_1\overline{T_w}=b_1\overline{T_sT_v}
=b_1(\overline{T_s}\,\overline{T_v})=(b_1\overline{T_s})\overline{T_v}\\
=(\overline{T_s}b_1\!)\overline{T_v}
=\overline{T_s}(b_1\overline{T_v})=\overline{T_s}(\underline{b_1T_v})
=\overline{T_s}\underline{b_v}
=\overline{T_s}\,\overline{b_v}=\overline{T_s}(\overline{T_vb_1})\\
=\overline{T_s}(\overline{T_v}b_1\!)
=(\overline{T_s}\,\overline{T_v})b_1
=\overline{T_sT_v}b_1=\overline{T_w}b_1=\overline{T_wb_1}=\overline{b_w}
\qquad
\end{multline*}
as required.
\qed
\end{proof}
So if \((\mathscr{I}\!,\,J,K)\) is a \(W\!\)-graph biideal then it is indeed
true that \(C={C\vtms}\opp\). Moreover, we also see that \(q_{y,w}\opp=q_{y,w}\) for
all \(y,\,w\in\mathscr{I}\) with \(y<w\), and hence
\(\mu_{y,w}\opp =\mu_{y,w}\) for all \(y,\,w\in\mathscr{I}\) with \(y<w\).
It follows from this that \(\Gamma=(C,\mu,\tau)\) is a \((W\times W\opp)\)-graph,
where \(\mu\) is defined by 
\begin{equation*}
\mu(c_{y},c_{w}) = \begin{cases}
                         \mu_{y,w} &\text{if \(y < w\)}\\
                         \mu_{w,y} &\text{if \(w < y\)}\\
                         0 &\text{otherwise},
                     \end{cases}
\end{equation*}
and \(\tau\) is defined by
\(\tau(c_{w}) = \D_J(w,\mathscr{I}) \sqcup \D_K(w^{-1},\mathscr{I}^{-1})\opp\)
for all~\(w\in\mathscr{I}\!\).

\begin{theorem}\label{extra-main-wg}
If \((\mathscr{I}\!,\,J,K)\) is a \(W\!\)-graph biideal, then the triple
\(\Gamma=(C, \mu, \tau)\) defined in the above preamble is a
\((W \times W^{\mathrm o})\)-graph.
\end{theorem}

\begin{remark}\label{regular biideal}
The work of Kazhdan and Lusztig \cite{kazlus:coxhecke} shows that \((W,\emptyset,\emptyset)\)
is a \(W\!\)-graph biideal.
\end{remark}

\begin{remark}\label{LRtoLandR}
With the notation as in Theorem~\ref{extra-main-wg}, let
\({\tau\vtms}\lside\colon C\to\powerset(S)\) and
\({\tau\vtms}\rside\colon C\to\powerset(S\opp)\) be defined by
\({\tau\vtms}\lside(c)=\tau(c)\cap S\) and \({\tau\vtms}\rside(c)=\tau(c)\cap S\opp\)
for all \(c\) in~\(C\), so that \(\Gamma\lside=(C,\mu,{\tau\vtms}\lside)\) is the
\(W\!\)-graph \(\Gamma(\mathscr{I}\!,\,J)\) and \(\Gamma\rside=(C,\mu,{\tau\vtms}\rside)\)
is the \(W\opp\!\)-graph \(\Gamma(\mathscr{I}\!,\,K)\). As in Section~\ref{sec:2}
above, the functions \(\mu\) and \(\tau\) determine a preorder \(\leqslant_\Gamma\) on~\(C\);
we call the corresponding equivalence classes the two-sided cells of~\(C\). Similarly
\(\Gamma\lside\) and \(\Gamma\rside\) yield preorders \(\leqslant_{\Gamma\lside}\) and
\(\leqslant_{\Gamma\rside}\) on~\(C\); the corresponding equivalence classes are called
the left cells and right cells of~\(C\).
\end{remark}

\begin{remark}
It is obvious from the definitions that if \((\mathscr{I}\!,\,J,K)\)
is a \(W\!\)-graph biideal then so is \((\mathscr{I}^{-1}\!,\,K,J)\). If
\(f\colon\mathscr{S}(\mathscr{I}\!,\,J,K)\to\mathscr{S}(\mathscr{I}^{-1}\!,\,K,J)\)
is the \(\mathcal{A}\)-isomorphism defined by \(f(b_w)=b_{w^{-1}}\) then
\(hf(b)=f(bh^\flat)\) and \(f(b)h=f(h^\flat b)\) for all \(b\in\mathscr{S}(\mathscr{I}\!,\,J,K)\)
and \(h\in\mathcal{H}\). Furthermore, for all \(y,\,w\in\mathscr{I}\!\), the polynomial
\(q_{y,w}\) for \((\mathscr{I}\!,\,J,K)\) equals the polynomial
\(q_{y^{\smash{-1}},w^{\smash{-1}}}\opp=q_{y^{-1},w^{-1}}\) for \((\mathscr{I}^{-1}\!,\,K,J)\).
So, in the important special case that \(\mathscr{I}=\mathscr{I}^{-1}\)
and \(J=K\), we have \(q_{y^{-1},w^{-1}}=q_{y,w}\) for all \(y,\,w\in\mathscr{I}\).
This corresponds to the well known identity \(P_{y^{-1},w^{-1}}=P_{y,w}\) for
Kazhdan--Lusztig polynomials, established in \cite[5.6]{lus:heckeuneqal}.
\end{remark}
In Definition~\ref{wgbiideal}, the requirement that \(\mathscr{S}\) is a
\((\mathcal{H},\mathcal{H})\)-bimodule is not implied by
other requirements, as the following example shows.

\begin{example}
Let \(W\) be the Weyl group of type \(A_{2}\), with \(S = \{s,t\}\). We shall show that
\((\mathscr{I}\!,\,J)=(\{1,t\},\{s\})\) is both a \(W\!\)-graph ideal and a \(W\!\)-graph
right ideal, but \((\mathscr{I}\!,\,J,J)\) is not a \(W\!\)-graph biideal.

Recall first that \(D_J=\{1,t,st\}\), and that \((D_J,J)\) is a \(W\!\)-graph ideal
(by \cite[Theorem~9.2]{howvan:wgraphDetSets}). Let \(C=\{c_1,c_t,c_{st}\}\) be the 
\(W\!\)-graph basis of the corresponding \(\mathcal{H}\)-module. Since \(s\) is a
strong descent of \(st\) and \(t\) is a weak descent of~\(st\), it follows
that \(T_sc_{st}=T_tc_{st}=-q^{-1}c_{st}\). So the set \(\{st\}\) is a
\((D_J,J)\)-closed subset of \(D_J\), and it follows by Theorem~\ref{wgdetset} that
\(\mathscr{I}\) is a (strong) \(W\!\)-graph subideal of \((D_J,J)\) (since
\(\mathscr{I}=D_J\setminus\{st\}\)). In particular, \((\mathscr{I}\!,\,J)\) is a
\(W\!\)-graph ideal. Since \(\mathscr{I}=\mathscr{I}^{-1}\) we conclude that
\((\mathscr{I}\!,\,J)\) is also a \(W\!\)-graph right ideal.

Suppose, for a contradiction, that \((\mathscr{I}\!,\,J,J)\) is a \(W\!\)-graph biideal,
and let \(\Gamma=(C,\mu,\tau)\) be the corresponding \((W\times W\opp)\)-graph, defined
as in the preamble to Theorem~\ref{extra-main-wg}. Thus \(C=\{c_1,c_t\}\) is an
\(\mathcal{A}\)-basis for \(M_\Gamma\), which is an \((\mathcal{H},\mathcal{H})\)-bimodule.
Since \(\D_J(\mathscr{I}\!,\,1)=J=\{s\}\) and \(\D_J(\mathscr{I}\!,\,t)=\{t\}\) it
follows that \(\tau(c_1)=\{s,s\opp\}\) and \(\tau(c_t)=\{t,t\opp\}\), and since it is
immediate from Corollary~\ref{recursion} that \(\mu_{1,t}=q_{1,t}=1\) we conclude that
\begin{alignat*}{5}
T_{s}c_{1} &= c_{1}T_{s} &{}&= -q^{-1}c_{1},&&& T_{t}c_{1} &= c_{1}T_{t} &{}&= qc_{1}+ c_{t},\\[-9 pt]
&&&&\text{\qquad\quad and\qquad\quad}\\[-9 pt]
T_{s}c_{t} &= c_{t}T_{s} &{}&= c_{1} + qc_{t},&&& T_{t}c_{t} &= c_{t}T_{t} &{}&= -q^{-1}c_{t}.
\end{alignat*}
The observation that
\((T_sc_1)T_t=-q^{-1}c_1T_t\ne T_s(c_1T_t)\)
gives the desired contradiction.
\end{example}

\begin{definition}\label{wgsubbiideal}
Suppose that \((\mathscr{I}\!,\,J,K)\) and \((\mathscr{I}_0,J_0,K_0)\) are \(W\!\)-graph
biideals. We say that \(\mathscr{I}\) is a \textit{\(W\!\)-graph subbiideal of
\(\mathscr{I}_0\)} if \(\mathscr{I}\subseteq\mathscr{I}_0\) and \((J,K)=(J_0,K_0)\).
\end{definition}
The following result is the biideal analogue of Theorem~\ref{wgdetset}.

\begin{theorem}\label{subbiideal}
Let \((\mathscr{I}_{0}, J,K)\) be a \(W\!\)-graph biideal with corresponding
\((W\times W\opp)\)-graph \(\Gamma=(C_0,\mu,\tau)\),
so that \(C_{0} = \{\,c^{0}_{w} \mid w \in \mathscr{I}_{0}\,\}\) is an
\(\mathcal{A}\)-basis of the bimodule \(\mathscr{S}_0=\mathscr{S}(\mathscr{I}_{0},J,K)\).
Let \(\mathscr I\subseteq\mathscr{I}_0\) be such that
\(\{\,c^{0}_{w} \mid w \in \mathscr I_{0} \setminus \mathscr I\,\} \subseteq C_{0}\)
is closed with respect to the (two-sided) preorder \(\leqslant_{\Gamma}\) on \(C_0\).
Then \((\mathscr I\!,\,J,K)\) is a \(W\!\)-graph biideal, and the
\((W\times W\opp)\)-graph \(\Gamma(\mathscr I\!,\,J,K)\) is isomorphic to the full
subgraph of \(\Gamma\) on the vertex set \(\{\,c_w^0\mid w\in\mathscr I\,\}\subseteq C_0\),
with \(\mu\) and \(\tau\) functions inherited from \(\Gamma\).
\end{theorem}

\begin{proof}
Since the set \(C'=\{\,c_{w}^{0} \mid w \in \mathscr{I}_{0} \setminus \mathscr{I}\,\}\)
is closed with respect to \(\leqslant_{\Gamma}\), it follows from the theory described
in Section~\ref{sec:2} that \(\mathcal{A}C'\) is an \((\mathcal{H},\mathcal{H})\)-bimodule,
and also that \(C'\) is closed with respect to the left and right preorders
\(\leqslant_{\Gamma\lside}\) and \(\leqslant_{\Gamma\rside}\) defined as in Remark~\ref{LRtoLandR}
above. Hence it follows from Theorem~\ref{wgdetset} that \((\mathscr{I}\!,\,J)\) is a
\(W\!\)-graph ideal and also that \((\mathscr{I}\!,\,K)\) is a \(W\!\)-graph right ideal.
Moreover, by Remark~\ref{remproof} the left \(\mathcal{H}\)-module
\(\mathscr{S}(\mathscr{I}\!,\,J)\) and the right \(\mathcal{H}\)-module
\(\mathscr{S}\opp(\mathscr{I}\!,\,K)\) can both be identified with 
\(\mathscr{S}_0/\mathcal{A}C'\) (which is an \((\mathcal{H},\mathcal{H})\)-bimodule),
with the standard basis of \(\mathscr{S}(\mathscr{I}\!,\,J)\) and that of
\(\mathscr{S}\opp(\mathscr{I}\!,\,K)\) both equal to \(\{\,f(b_w^0)\mid w\in\mathscr{I}\,\}\),
where \(\{\,b_w\mid w\in\mathscr{I}_{0}\,\}\) is the standard basis of \(\mathscr{S}_0\)
and \(f\) is the natural map \(\mathscr{S}_0\to \mathscr{S}_0/\mathcal{A}C'\).
Hence \((\mathscr I\!,\,J,K)\) is a \(W\!\)-graph biideal, by Definition~\ref{wgbiideal}.
The remaining assertions follow from Theorem~\ref{wgdetset} and its right ideal analogue
applied to \((\mathscr{I}_{0}, J)\) and~\((\mathscr{I}_{0},K)\).
\qed
\end{proof}

\begin{remark}
Let \((\mathscr{I}\!,\, J,K)\) be a \(W\)-graph biideal and
\(C=\{\,c_w\mid w\in\mathscr{I}\,\}\) the \((W\times W\opp)\)-graph basis of
\(\Gamma=\Gamma(\mathscr{I}\!,\, J,K)\). In keeping with the conventions
we adopted in the preamble to Proposition~\ref{cellsandstrongsubideals} above,
we say that a subset \(X\) of \(\mathscr{I}\) is \((\mathscr{I}\!,\,J,K)\)-closed
if \(\{\,c_x\mid x\in X\,\}\) is closed with respect to the preorder \(\leqslant_{\Gamma}\),
and call \(X\) a two-sided cell of \((\mathscr{I}\!,\,J,K)\) if \(\{\,c_x\mid x\in X\,\}\)
is a cell of~\(\Gamma\). Clearly \(\leqslant_{\Gamma}\) induces a partial ordering 
on the set of two-sided cells, and \(X\subseteq\mathscr{I}\) is
\((\mathscr{I}\!,\,J,K)\)-closed if and only if it is a union of two-sided that
form an ideal with respect to this order. Theorem~\ref{subbiideal} shows that
the complement in \(\mathscr{I}\) of any such union is a \(W\!\)-graph
biideal with respect to \(J\)~and~\(K\).
\end{remark}

\section{Computational characterization of \textit{W-}graph ideals}
\label{sec:5}
Let \((W,S)\) be a Coxeter system, \(\mathscr{I}\) an ideal of \((W,\leqslant\lside)\) and
\(J\) a subset of \(\Pos(\mathscr{I})\). We know that if \((\mathscr{I},J)\) is a 
\(W\!\)-graph ideal then we can construct an \(\mathcal{H}\)-module that has an
\(\mathcal{A}\)-basis \(\{\,c_w\mid w\in\mathscr{I}\,\}\) on which the
generators of \(\mathcal{H}\) via the formulas given in Theorem~\ref{main1},
where the parameters~\(\mu_{y,w}\) are the constant terms of a family of
polynomials~\(q_{y,w}\) that can be computed recursively using the formulas
in Corollary~\ref{recursion}. In this section we prove the converse: if 
\((\mathscr{I},J)\) gives rise to an \(\mathcal{H}\)-module via this
construction then \((\mathscr{I},J)\) must be a \(W\!\)-graph ideal.

Note that if \((\mathscr{I},J)\) is not a \(W\!\)-graph ideal then the polynomials
\(q_{y,w}\) are not necessarily uniquely determined by the formulas in
Corollary~\ref{recursion}. If \(z\in\mathscr{I}\) and the \(q_{y,w}\) have been
found for all \(y,w\in\mathscr{I}\) with \(y<w<z\), then computing the
polynomials \(q_{y,z}\) involves first choosing some \(s\in\SD(z)\), so that
\(z=sw\) with \(w<z\), after which the formulas for \(q_{y,sw}\) can be applied.
A different sequence of choices of the elements \(s\in\SD(z)\) could conceivably produce
a different family of polynomials. We show that if some sequence of choices produces
polynomials that give rise to an \(\mathcal{H}\)-module then \((\mathscr{I},J)\) must
be a \(W\!\)-graph ideal. So, to be precise, our assumptions are as follows:
\begin{itemize}
\item[(A1)] \(\mathscr{I}\) is an ideal of \((W,\leqslant\lside)\) and
\(J\subseteq\Pos(\mathscr{I})\),
and \(\mathscr{S}\) is an \(\mathcal{A}\)-free \(\mathcal{H}\)-module;
\item[(A2)] \(\mathscr{S}\) has an \(\mathcal{A}\)-basis \(C=\{\,c_w\mid w\in\mathscr{I}\,\}\)
in bijective correspondence with~\(\mathscr{I}\!\), such that for certain integers~\(\mu_{y,w}\)
\[
T_sc_{w}=
\begin{cases}
-q^{-1}c_{w}&\text{if \(s\in\D(w)\),}\\
qc_{w}+\sum_{y\in \mathcal R(s,w)}\mu_{y,w}c_{y}
&\text{if \(s\in\WA(w)\),}\\
qc_{w}+c_{sw}+\sum_{y\in \mathcal R(s,w)}\mu_{y,w}c_{y} &\text{if
\(s\in\SA(w)\),}
\end{cases}
\]
where the set \(\mathcal R(s,w)\) consists of all \(y\in\mathscr{I}\) such that
\(y < w\) and \(s\in\D(y)\);
\item[(A3)] there exist polynomials \(q_{y,w}\in\mathcal{A}^+\), defined whenever
\(y,\,w\in\mathscr{I}\!\), such that \(\mu_{y,w}\) is the
constant term of \(q_{y,w}\), and \(q_{y,w}=0\) whenever \(y\not<w\);
\item[(A4)] for each \(z\in\mathscr{I}\) with \(z\ne 1\) there exists
\(s\in S\) with \(l(sz)<l(z)\) such that \(q_{sz,z}=1\), and
for all \(y\in\mathscr{I}\) with \(y<z\) we have
\begin{itemize}[itemsep=1 pt, itemindent=12 pt, topsep=1 pt]
\item[(1)] \(q_{y,z}=qq_{y,sz}\) \ if \(s\in\A(y)\),
\item[(2)]
\(q_{y,z} = -q^{-1}(q_{y,sz}-\mu_{y,sz})+q_{sy,sz}
+\sum_x\mu_{y,x}q_{x,sz}\) \ if \(s\in\SD(y)\),
\item[(3)]
\(q_{y,z} = -q^{-1}(q_{y,sz}-\mu_{y,sz})+\sum_x\mu_{y,x}q_{x,sz}\) \
if \(s\in\WD(y)\),
\end{itemize}
where the sums in (2) and (3) extend over all \(x\in \mathscr I\) such
that \(y<x<sz\) and \(s\notin\D(x)\).
\end{itemize}
The conclusion is that \((\mathscr{I},J)\) is a \(W\!\)-graph ideal.
The proof consists of showing that the module \(\mathscr{S}\) 
satisfies the conditions of Definition \ref{wgphdetelt}.

Since \(C\) is an \(\mathcal{A}\)-basis of \(\mathscr{S}\) there is an
\(\mathcal{A}\)-semilinear involution \(\alpha\mapsto\overline\alpha\) on
\(\mathscr{S}\) such that \(\overline{c_w}=c_w\) for all \(w\in\mathscr{I}\!\).
Since \(\overline{T_s-q}=T_s-q\) and \(\overline{T_s+q^{-1}}=T_s+q^{-1}\)it follows from assumption (A2) that
\(\overline{T_s}\overline{c_w}=\overline{T_s c_w}\) in each of the three cases,
and hence \(\overline{h\alpha}=\overline{h}\overline{\alpha}\) for all
\(h\in\mathcal{H}\) and \(\alpha\in\mathscr{S}\!\). The remaining task is to show that
\(\mathscr{S}\) has an \(\mathcal{A}\)-basis \(\{\,b_w\mid w\in\mathscr{I}\,\}\)
such that the formulas in~Eq.~\eqref{S_0action} hold. We define
\(b_w=T_wc_1\) for all \(w\in\mathscr{I}\!\), and observe first that
Eq.~\eqref{S_0action} is satisfied in three of the four cases.

\begin{proposition}\label{easybit}
Let \(w\in\mathscr{I}\) and \(s\in S\), and suppose that \(s\notin\WA(w)\). Then
\[
T_{s}b_{w} =
\begin{cases}
  b_{sw}  & \text{if \(s \in \SA(w)\),}\\
  b_{sw} + (q - q^{-1})b_{w} & \text{if \(s \in \SD(w)\),}\\
  -q^{-1}b_{w} & \text{if \(s \in \WD(w)\).}
\end{cases}
\]
\end{proposition}
\begin{proof}
If \(s\in\SA(w)\) then \(w<sw\in\mathscr{I}\), by the definition of \(\SA(w)\),
and by the definition of \(b_w\) and \(b_{sw}\) it follows that
\(T_sb_w=T_s(T_wc_1)=(T_sT_w)c_1=T_{sw}c_1=b_{sw}\), as required.

If \(s\in\SD(w)\) then \(s\in\SA(sw)\), and so from the case we have just done it follows
that
\(T_sb_w=T_s(T_sb_{sw})=T_s^2b_{sw}=(1+(q-q^{-1})T_s)b_{sw}=b_{sw}+(q-q^{-1})b_{w}\),
as required.

Now suppose that \(s\in\WD(w)\). Since this gives \(w\in D_J\) and \(sw\notin D_J\),
it follows from Lemma~\ref{deo1} that \(l(sw)=l(w)+1\) and \(sw=wt\) for some \(t\in J\).
So \(T_sT_w=T_{sw}=T_{wt}=T_wT_t\). Furthermore, \(T_tc_1=-q^{-1}c_1\), since \(t\in J=\WD(1)\).
Hence
\[
T_sb_w=T_s(T_wc_1)=(T_sT_w)c_1=(T_wT_t)c_1=T_w(T_tc_1)=-q^{-1}T_wc_1=-q^{-1}b_w,
\]
as required.\qed
\end{proof}

\begin{lemma}\label{uniBbasis1}
We have \(b_{z} = c_{z} + q\sum_{\{y\in\mathscr{I}\mid y < z\}} q_{y,z}c_{y}\)
for all \(z\in\mathscr{I}\!\).
\end{lemma}
\begin{proof}
The proof is by induction on~\(l(z)\), the case \(l(z)=0\) being trivial.
So we assume that \(l(z)>1\), and choose \(s\) as in assumption (A4) above.
We write
\begin{align*}
\mathcal{R}&=\{\,x\in\mathscr{I}\mid\,x<sz\text{ and }s\in\D(x)\,\},\\
\mathcal{T}_1&=\{\,x\in\mathscr{I}\mid\,x<sz\text{ and }s\in\SA(x)\,\},\\
\mathcal{T}_2&=\{\,x\in\mathscr{I}\mid\,x<sz\text{ and }s\in\WA(x)\,\},
\end{align*}
so that \(\mathcal{R}\) is the set \(\mathcal{R}(s,sz)\) of assumption~(A2) above,
and we also write \(\mathcal{T}=\mathcal{T}_1\cup\mathcal{T}_2\).
The inductive hypothesis gives \(b_{sz}=c_{sz} + q\sum_{x < sz} q_{x,sz}c_{x}\),
and Proposition~\ref{easybit} gives \(b_z=T_sb_{sz}\), since \(s\in\SA(sz)\). So,
using (A2) to evaluate \(T_sc_{sz}\) and \(T_sc_x\) for \(x\in\mathcal{R}\),
\begin{align*}
b_z&=T_sc_{sz}+q\sum_{x\in\mathcal{R}}q_{x,sz}T_sc_{x}
+q\sum_{x\in\mathcal{T}_1}q_{y,sz}T_sc_{x}
+q\sum_{x\in\mathcal{T}_2}q_{y,sz}T_sc_{x}\\
&=(c_z+qc_{sz}+\sum_{x\in\mathcal{R}}\mu_{x,sz}c_x)
-\sum_{y\in\mathcal{R}}q_{x,sz}c_{x}
+q\sum_{x\in\mathcal{T}_1}q_{x,sz}T_sc_{x}
+q\sum_{x\in\mathcal{T}_2}q_{y,sz}T_sc_{x}\\
&=c_z+qc_{sz}-\sum_{x\in\mathcal{R}}(q_{x,sz}-\mu_{x,sz})c_x
+q\sum_{x\in\mathcal{T}_1}q_{x,sz}T_sc_{x}
+q\sum_{x\in\mathcal{T}_2}q_{x,sz}T_sc_{x}.
\end{align*}
Now using (A2) to evaluate \(T_sc_x\) for \(x\in\mathcal{T}_1\) and
\(x\in\mathcal{T}_2\), and making use of the similarity between the two formulas,
we find that
\[
b_z-c_z=qc_{sz}-\sum_{x\in\mathcal{R}}(q_{x,sz}-\mu_{x,sz})c_x
+q\sum_{x\in\mathcal{T}_1}q_{x,sz}c_{sx}
+q\sum_{x\in\mathcal{T}}q_{x,sz}\Bigl(qc_x+\!\!\sum_{y\in\mathcal{R}(s,x)}\!\!\mu_{y,x}c_y\Bigr).
\]
We proceed to collect the coefficients of the various elements of \(C\) in the
right hand side. Note first that if \(x\in\mathcal{T}_1\) then \(sx\in\mathscr{I}\)
(since \(s\in\SA(x)\)), and Lemma~\ref{lifting1} implies that \(sx<z\), since \(x<sz<z\).
So all the elements of \(C\) that appear have the form \(c_y\) with~\(y<z\).
Writing \(\coeff(y)\) for the coefficient of~\(c_y\), the aim is to show that
\(\coeff(y)=qq_{y,z}\).

Let \(y\in\mathscr{I}\) with \(y<z\), and suppose first that \(s\in\A(y)\). Then \(y<sy\),
and so \(y\leqslant sz\) by Lemma~\ref{lifting1}. So either \(y=sz\) and
\(\coeff(y)=q\), or else \(y\in\mathcal{T}\) and \(\coeff(y)=q^2q_{y,sz}\).
In either case \(\coeff(y)=qq_{y,z}\), by assumption~(A4).

Now suppose that \(s\in\WD(y)\). Then \(y\notin\{\,sx\mid x\in\mathcal{T}_1\,\}\),
since \(sy\notin\mathscr{I}\!\). So \(c_y\) occurs only in the the first sum in
our expression and in the double sum. Hence
\[
\coeff(y)=-(q_{y,sz}-\mu_{y,sz})+\sum_x q\mu_{y,x}q_{x,sz}
\]
where \(x\) runs through all elements of \(\mathcal{T}\) such that \(y\in\mathcal{R}(s,x)\).
Again we see from assumption (A4) that \(\coeff(y)=qq_{y,z}\).

Finally, suppose that \(s\in\SD(y)\). In this case \(y=sx\) with \(x\in\mathcal{T}_1\),
so that we obtain a term \(q_{sy,sz}c_y\) in addition to the terms obtained in the
case \(s\in\WD(y)\). So again \(\coeff(y)=qq_{y,z}\), as required.\qed 
\end{proof}
The following result completes the proof that Eq.~\eqref{S_0action} is satisfied.

\begin{proposition}
Let \(w\in\mathscr{I}\) and \(s\in\WD(w)\). Then
\(T_sb_w=qb_w+\sum_{\{y\in\mathscr{I}\mid y<sw\}}r^s_{y,w}b_y\) for
some polynomials \(r^s_{y,w}\in q\mathcal{A}^+\).
\end{proposition}

\begin{proof}
Define \(\mathcal{R}=\{\,y\in\mathscr{I}\mid\,y<w\text{ and }s\in\D(y)\,\}\), 
so that \(\mathcal{R}=\mathcal{R}(s,w)\), and define
also \(\mathcal{T}_1=\{\,y\in\mathscr{I}\mid\,y<w\text{ and }s\in\SA(y)\,\}\) and
\(\mathcal{T}_2=\{\,y\in\mathscr{I}\mid\,y<w\text{ and }s\in\WA(y)\,\}\).
In addition, let \(\mathcal{T}=\mathcal{T}_1\cup\mathcal{T}_2\).
Since \(b_w=c_w+q\sum_{y<w}q_{y,w}c_y\) we see from assumption (A2) that
\begin{align*}
T_sb_w&=T_sc_w+\sum_{y\in\mathcal{R}}qq_{y,w}T_sc_y+\sum_{y\in\mathcal{T}_1}qq_{y,w}T_sc_y
+\sum_{y\in\mathcal{T}_2}qq_{y,w}T_sc_y\\
&=(qc_w+\sum_{y\in\mathcal{R}}\mu_{y,w}c_y)-\sum_{y\in\mathcal{R}}q_{y,w}c_y
+\sum_{y\in\mathcal{T}_1}qq_{y,w}c_{sy}+
\sum_{y\in\mathcal{T}}qq_{y,w}\Bigl(qc_y+\!\!\sum_{x\in\mathcal{R}(s,y)}\!\!\mu_{x,y}c_x\Bigr)\\
&=qc_w-\sum_{y\in\mathcal{R}}(q_{y,w}-\mu_{y,w})c_y
+\sum_{y\in\mathcal{T}_1}qq_{y,w}c_{sy}+
\sum_{y\in\mathcal{T}}qq_{y,w}\Bigl(qc_y+\!\!\sum_{x\in\mathcal{R}(s,y)}\!\!\mu_{x,y}c_x\Bigr).
\end{align*}
Since \(\mu_{y,w}\) is the constant term of \(q_{y,w}\), every element of \(C\) appearing in
the above expression has coefficient lying in \(q\mathcal{A}^+\!\). So, using
Lemma~\ref{lifting1} and the fact that \(w<sw\) (since \(s\in\WA(w)\)), it follows that
\begin{equation}\label{eq:Tsbw}
T_sb_w=\sum_{x<sw}t_{x,w}c_x\qquad\text{for some \(t_{x,w}\in q\mathcal{A}^+\).}
\end{equation}
Inverting the system of equations in Lemma~\ref{uniBbasis1} shows that for
all \(x\in\mathscr{I}\) there exist \(p_{y,x}\in\mathcal{A}^+\) such that
\(c_x=b_x-q\sum_{y<x}p_{y,x}b_y\), and substituting this into Eq.~\eqref{eq:Tsbw}
gives the required result, with \(r^s_{y,w}=t_{y,w}-q\sum_{\{x|y<x<sw\}} p_{y,x}t_{x,w}\).
\qed
\end{proof}
We have now shown that all the requirements of Definition~\ref{wgphdetelt} are satisfied,
and so \((\mathscr{I}\!,\,J)\) is a \(W\!\)-graph ideal. So we have proved the
following theorem.

\begin{theorem}\label{converseMain}
Let \(\mathscr{I}\) be an ideal of \((W,\leqslant\lside)\) and \(J\subseteq\Pos(\mathscr{I})\).
Then \((\mathscr{I}\!,\,J)\) is a \(W\!\)-graph ideal if and only if the
construction described in Section~\ref{sec:3} above produces a \(W\!\)-graph
\((C,\mu,\tau)\) such that Theorem~\ref{main1} is satisfied.
\end{theorem}

\begin{remark}
According to the construction, \(C=\{\,c_w\mid w\in\mathscr{I}\,\}\) and
\(\tau(w)=\D_J(\mathscr{I}\!,\,w)\) for all \(w\in\mathscr{I}\). The function
\(\mu\) is defined as in Eq.~\eqref{mu-symmetric}, where \(\mu_{y,w}\)
is the constant term of \(q_{y,w}\), and these polynomials satisfy the formulas
in Corollary~\ref{recursion}. In fact we showed that if \((C,\mu,\tau)\) is
a \(W\!\)-graph then the conclusion that \((\mathscr{I}\!,\,J)\) is a \(W\!\)-graph
ideal needs only the weaker assumption that the \(q_{y,w}\) are computed using~(A4)
above. Given that \((C,\mu,\tau)\) is a \(W\!\)-graph, it is not hard to show that
Theorem~\ref{main1} is satisfied if and only if the statement of
Corollary~\ref{ywmorethan3} holds.
\end{remark}
To conclude this section we give an example of an ideal \(\mathscr{I}\) of \((W,\leqslant\lside)\)
and a subset \(J\) of \(\Pos(\mathscr{I})\) such that \((\mathscr{I}\!,\,J)\) is not a
\(W\!\)-graph ideal, but nevertheless has the property that
there exists a \(W\!\)-graph \((C,\mu,\tau)\) with \(C=\{\,c_w\mid w\in\mathscr{I}\,\}\)
and \(\tau(c_w)=\D_J(\mathscr{I}\!,\,w)\) for all \(w\in\mathscr{I}\!\).

\begin{example}
\label{ex:1}
Let \((W,S)\) be the Coxeter system of type \(B_{4}\), and let
\(S = \{s_{0},s_{1},s_{2},s_{3}\}\), where \(s_0s_1\) has order~4 and
\(s_1s_2\) and \(s_2s_3\) have order~3.
Let \(\mathscr I = \{1, s_0, s_{1}s_0, s_{2}s_{1}s_0\}\) and note that
\(\mathscr{I}\subseteq D_J\), where \(J=\{s_{1},s_{2}, s_{3}\}\). We
use Theorem~\ref{converseMain} to determine whether or not
\((\mathscr{I},J)\) is a \(W\!\)-graph ideal. The first step is to
compute the polynomials \(q_{y,w}\),
for all \(y,w \in \mathscr I\) with \(y < w\), using the formulas
given in Corollary \ref{recursion} (or (A4) above).

It is immediate that the three cases with \(l(w) - l(y) = 1\) give \(q_{y,w} = 1\).
For the next case, let \((y,w) = (1,s_{1}s_0)\), and observe that \(s_1\) is the only
strong descent of~\(w\). Since \(s_{1} \in \WD_J(y)\), the third formula of
Corollary~\ref{recursion} applies, and gives
\(q_{1,s_{1}s_0} = q^{-1}(q_{1,s_0} - \mu_{1,s_0}) = 0\).
There are now two remaining possibilities for \((y,w)\), both with
\(w=s_{2}s_{1}s_0\). Observe that \(s_2\) is the only
strong descent of~\(w\), and \(s_2\in\WD_J(y)\) for both values
of~\(y\), namely \(y=s_0\) and \(y=1\). Furthermore, in both cases
\(\{x\in\mathscr{I} \mid y < x < s_{1}s_0 \text{ and } s_{2} \notin \D(x)\}\)
is empty, and so it follows that \(q_{y,w}=q^{-1}(q_{y,s_1s_0}-\mu_{y,s_1s_0})=0\).
So the graph obtained is
\[
\begin{tikzpicture}[scale=0.45][font=\small]
\draw (0,0) circle [radius=0.6];
\node at (0,0) {\(\scriptstyle{1,2,3}\)};
\draw [<->] (0.7,0) --(2.3,0);
\draw (3,0) circle [radius=0.6];
\node at (3,0) {\(\scriptstyle{0,2,3}\)};
\draw [<->] (3.7,0) --(5.3,0);
\draw (6,0) circle [radius=0.6];
\node at (6,0) {\(\scriptstyle{1,3}\)};
\draw [<->] (6.7,0) --(8.3,0);
\draw (9,0) circle [radius=0.6];
\node at (9,0) {\(\scriptstyle{1,2}\)};
\end{tikzpicture}
\]
where the numbers in the circles give the values of \(\D_J(w)\) for the various
elements~\(w\in\mathscr{I}\!\), and the edges all have weight~1.

It is easily checked that the above graph is not a \(W\!\)-graph: the relation
\(T_{s_0}T_{s_3}=T_{s_3}T_{s_0}\) fails. So \((\mathscr{I},J)\) is
not a \(W\!\)-graph ideal. However, adding an edge of weight \(-1\) joining the
vertices \(1\) and \(s_2s_1s_0\) gives
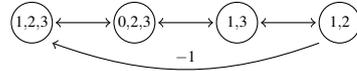
\[
\begin{tikzpicture}[scale=0.45][font=\small]
\draw (0,0) circle [radius=0.6];
\node at (0,0) {\(\scriptstyle{1,2,3}\)};
\draw [<->] (0.7,0) --(2.3,0);
\draw (3,0) circle [radius=0.6];
\node at (3,0) {\(\scriptstyle{0,2,3}\)};
\draw [<->] (3.7,0) --(5.3,0);
\draw (6,0) circle [radius=0.6];
\node at (6,0) {\(\scriptstyle{1,3}\)};
\draw [<->] (6.7,0) --(8.3,0);
\draw (9,0) circle [radius=0.6];
\node at (9,0) {\(\scriptstyle{1,2}\)};
\draw [<-](0.6,-0.6) to [out=340,in=200] (8.4,-0.6);
\node at (4.5,-1) {\(\scriptstyle{-1}\)};
\end{tikzpicture}
\]
and it is easily checked that this is a \(W\!\)-graph for which the
formulas in Theorem~\ref{main1} hold.
\end{example}

\section{Parabolic restriction}
\label{sec:6}
Let \((\mathscr I,J)\) be a \(W\!\)-graph ideal and let \(K \subseteq S\).
Let \(\mathcal{H}_{K}\) be the subalgebra of \(\mathcal{H}\) generated
by \(\{T_{s} \mid s \in K\}\). In this section we investigate the restriction
of \(\mathscr{S}(\mathscr{I}\!,\,J)\) to \(\mathcal{H}_K\). (As we noted in
Section~\ref{sec:1} above, \(\mathcal{H}_{K}\) can be identified with the Hecke
algebra of the Coxeter system \((W_K,K)\).) Let
\(\{\,b_w\mid w\in\mathscr{I}\,\}\) be the standard basis of
\(\mathscr{S}(\mathscr{I}\!,\,J)\) and \(\{\,c_w\mid w\in\mathscr{I}\,\}\) the
\(W\!\)-graph basis.

Each element \(w\in W\) has a unique factorization \(w = vd\) with
\(v \in W_{K}\) and \(d \in D_{K}^{-1}\). Since \(l(w)=l(v)+l(d)\) necessarily
holds in this situation, it follows that \(d\leqslant\lside w\). So
\(d\in\mathscr{I}\) whenever~\(w\in\mathscr{I}\!\). For each
\(d \in D_{K}^{-1}\) define \(\mathscr{I}_d\subseteq W_K\)
by \(\mathscr{I}_d = \{\,v\in W_{K}\mid vd \in \mathscr {I}\,\}\), so that
\begin{equation}\label{mackeyRestriction}
\mathscr I = \!\!\bigsqcup_{d \in D_{K}^{-1} \cap \mathscr I}\!\!\mathscr{I}_{d}d.
\end{equation}
and \(\mathscr{I}_{d}d=W_Kd\cap\mathscr{I}\) in each case.
Note that since \(\mathscr{I}\subseteq D_J\), each \(d\) appearing in
Eq.~\eqref{mackeyRestriction} is in \(D_{K,J}=D_K^{-1}\cap D_J\), the set
of minimal \((W_K,W_J)\) double coset representatives.

\begin{lemma}\label{Id-ideal}
Let \(d \in D_{K}^{-1} \cap \mathscr I\!\). Then \(\mathscr{I}_d\)
is an ideal of \((W_K\leqslant\lside)\), and \(K\cap dJd^{-1}\subseteq\Pos(\mathscr{I}_d)\).
\end{lemma}

\begin{proof}
Let \(w\in\mathscr{I}_d\) 
and let \(v\in W_K\) with \(v\leqslant\lside w\), so that \(w=uv\) with
\(l(w)=l(u)+l(v)\). Since \(v,\,w\in W_K\) and \(d\in D_K^{-1}\) we have
\(l(wd)=l(w)+l(d)\) and \(l(vd)=l(v)+l(d)\). Hence \(wd=u(vd)\) and
\(l(wd)=l(w)+l(d)=l(u)+l(v)+l(d)=l(u)+l(vd)\).
Since \(wd\in\mathscr{I}\) (since \(w\in\mathscr{I}_d\)) it follows that
\(vd\in\mathscr{I}\!\), and hence that \(v\in\mathscr{I}_d\). So \(\mathscr{I}_d\)
is an ideal of \((W_K\leqslant\lside)\).

Now let \(v\in\mathscr{I}_d\), so that \(v\in W_K\) and \(vd\in\mathscr{I}\!\), and
let \(s\in K\cap dJd^{-1}\), so that \(s\in K\) and \(sd=dr\) for some \(r\in J\). Since
\(J\subseteq\Pos(\mathscr{I})\) it follows that \(l((vd)r)>l(vd)\), and since
\(d\in D_K^{-1}\) and \(v,\,vs\in W_K\) we find that
\(l(vs)+l(d)=l(vsd)=l(vdr)>l(vd)=l(v)+l(d)\). Hence \(l(vs)>l(v)\), and we
conclude that \(K\cap dJd^{-1}\subseteq\Pos(\mathscr{I}_d)\). 
\qed
\end{proof}
For each \(d\in D_{K}^{-1}\cap\mathscr{I}\) let \(\mathscr{J}_d\subseteq\mathscr{I}\)
be defined by \(\mathscr{J}_d=\bigcup_e\mathscr{I}_ee\), where \(e\) runs through
the set \(\{\,e\in D_K^{-1}\mid e\leqslant d\,\}\), and let
\(\mathscr{J}_d'=\mathscr{J}_d\setminus\mathscr{I}_dd\). Let \(\mathscr{S}_d\) and
\(\mathscr{S}_d'\) be the \(\mathscr{A}\)-submodules of \(\mathscr{S}(\mathscr{I}\!,\,J)\)
spanned by \(\{\,c_w\mid w\in \mathscr{J}_d\,\}\) and
\(\{\,c_w\mid w\in \mathscr{J}_d'\,\}\) respectively.
Thus \(\mathscr{S}_d'\subseteq\mathscr{S}_d\), and the 
quotient module \(\mathscr{S}=\mathscr{S}_d/\mathscr{S}_d'\) has \(\mathcal{A}\)-basis
\(\{\,f(c_{wd})\mid w\in\mathscr{I}_d\,\}\), where \(f\) is the natural 
homomorphism \(\mathscr{S}_d\to \mathscr{S}\).

Clearly \(\mathscr{S}_d\) and \(\mathscr{S}_d'\) are both stable under the bar
involution of \(\mathscr{S}(\mathscr{I}\!,\,J)\), since \(\overline{c_w}=c_w\) for
all \(w\in\mathscr{I}\). Hence \(\mathscr{S}\) admits a
bar involution such that \(\overline{f(\alpha)}=f(\overline{\alpha})\) for all
\(\alpha\in\mathscr{S}_d\).

\begin{lemma}\label{bruhat-smaller}
Let \(y,\,w\in\mathscr{I}\) with \(y\leqslant w\), and suppose that
\(d\in D_{K}^{-1}\cap\mathscr{I}\). If \(w\in\mathscr{J}_d\) then \(y\in\mathscr{J}_d\), 
and if \(w\in\mathscr{J}_d'\) then \(y\in\mathscr{J}_d'\).
\end{lemma}

\begin{proof}
Let \(y\in W_Ke\) and \(w\in W_Ke'\), where \(e,\,e'\in D_K^{-1}\). 
Since \(y\leqslant w\) it follows that \(e\leqslant e'\), by Proposition~\ref{orderPresr}.
If \(w\in\mathscr{J}_d\) then we have \(e'\leqslant d\), by the definition of
\(\mathscr{J}_d\), so that \(e\leqslant d\) and \(y\in\mathscr{I}_{e}e\subseteq\mathscr{J}_d\).
If \(w\in\mathscr{J}_d'\) then \(e'<d\), giving \(e<d\) and \(y\in\mathscr{J}_d'\).
\qed
\end{proof}
The following lemma is the key result in this section.

\begin{lemma}\label{Jd-closed}
Let \(d\in D_K^{-1}\cap\mathscr{I}\). Then \(\mathscr{S}_d\) and
\(\mathscr{S}_d'\) are both \(\mathcal{H}_K\)-submodules of
\(\mathscr{S}(\mathscr{I}\!,\,J)\).
\end{lemma}

\begin{proof}
Let \(w\in\mathscr{J}_d\), so that \(w\in\mathscr{I_e}e=W_Kd\cap\mathscr{I}\) for some
\(e\in D_K^{-1}\) with \(e\leqslant d\), and let \(s\in K\). If \(sw\in\mathscr{I}\)
then \(sw\in\mathscr{I_e}e\subseteq\mathscr{J}_d\), since \(sw\in sW_Kd=W_Kd\).
If \(y\in\mathscr{I}\) and \(y<w\) then \(y\in\mathscr{J}_d\), by
Lemma~\ref{bruhat-smaller}. By Theorem~\ref{main1} we see that
\(T_sc_w\) is an \(\mathcal{A}\)-linear combination of terms that
all lie in  \(\{\,c_w\mid w\in \mathscr{J}_d\,\}\). So it follows that this
set spans an \(\mathcal{H}_K\)-submodule of \(\mathscr{S}(\mathscr{I}\!,\,J)\).
The proof of the other part is the same, with \(\mathscr{J}_d\) replaced
by \(\mathscr{J}_d'\).
\qed
\end{proof}
Observe that if \(d\in D_K^{-1}\cap\mathscr{I}\) and \(w\in\mathscr{J}_d\) then
\(b_w\in\mathscr{S}_d\), since \(b_w=c_w+q\sum_{y<w}q_{y,w}c_y\), and
Lemma~\ref{bruhat-smaller} shows that each \(y\) involved is in \(\mathscr{J}_d\).
The same applies with \(\mathscr{J}_d\) replaced by \(\mathscr{J}_d'\) and 
\(\mathscr{S}_d\) by \(\mathscr{S}_d'\). It follows that the sets
\(\{\,b_w\mid w\in\mathscr{J}_d\,\}\) and \(\{\,b_w\mid w\in\mathscr{J}_d'\,\}\)
are \(\mathcal{A}\)-bases of \(\mathscr{S}_d\) and \(\mathscr{S}_d'\),
and \(\{\,f(b_{wd})\mid w\in\mathscr{I}_d\,\}\) is an \(\mathcal{A}\)-basis
of \(\mathscr{S}\).

We are now able to prove the main result of this section.

\begin{theorem}\label{restrictedWGideal}
Let \((\mathscr I,J)\) be a \(W\!\)-graph ideal.
Suppose that \(K \subseteq S\) and \(d \in D_{K}^{-1} \cap \mathscr I\!\). Then
\(\mathscr{I}_d = \{v \in W_K \mid vd \in \mathscr I\}\) is a \(W_{K}\)-graph ideal
with respect to \(L = K \cap dJd^{-1}\).
\end{theorem}

\begin{proof}
It was proved in Lemma~\ref{Id-ideal} that \(\mathscr{I}_d\) is
an ideal of \((W_K,\leqslant\lside)\) and that \(L\subseteq\Pos(\mathscr{I}_d)\).
We proceed to show that Definition~\ref{wgphdetelt} is satisfied with
\(\mathscr{S}\) as \(\mathscr{S}(\mathscr{I}_d,K\cap dJd^{-1})\)
and with \(\{\,f(b_{wd})\mid w\in\mathscr{I}_d\,\}\) as its standard basis
(where, as above, \(\mathscr{S}=\mathscr{S}_d/\mathscr{S}_d'\)
and \(f\colon \mathscr{S}_d\to\mathscr{S}\) is the natural map).

Note that \(f(b_d)=f(c_d)\), since \(f(c_y)=0\) for all \(y\in\mathscr{I}\) with
\(y<d\). Hence \(\overline{f(b_d)}=f(b_d)\), and since also
\[
\overline{hf(\alpha)}=\overline{f(h\alpha)}=f(\overline{h\alpha})
=f(\overline{h}\overline{\alpha})=\overline{h}f(\overline{\alpha})
=\overline{h}\,\overline{f(\alpha)}
\]
for all \(\alpha\in\mathcal{S}_d\) and \(h\in\mathcal{H}_K\), it follows that
condition (ii) in Definition~\ref{wgphdetelt} is satisfied.
It remains to check that the generators \(T_s\) of \(\mathcal{H}_K\)
act on the basis elements \(f(b_{wd})\) in accordance with the requirements
of Eq.~\eqref{S_0action}.

Let \(s\in K\) and \(w\in\mathscr{I}_d\), and suppose first that
\(s\in\SA(\mathscr{I}_d,w)\). Then \(l(sw)>l(w)\) and \(sw\in\mathscr{I}_d\).
So \(s(wd)=(sw)d\in\mathscr{I}\!\), and \(l(s(wd))=l(sw)+l(d)>l(w)+l(d)=l(wd)\).
So \(s\in\SA(\mathscr{I}\!,\,wd)\), and so \(T_sb_{wd}=b_{s(wd)}\). Applying
\(f\) to both sides gives \(T_sf(b_{wd})=f(b_{(sw)d})\), as required.

Suppose next that \(s\in\SD(\mathscr{I}_d,w)\). Then \(s\in\SA(\mathscr{I}_d,sw)\),
and by the case just done we see that
\(T_sf(b_{wd})=T_s^2f(b_{(sw)d})=(1+(q-q^{-1})T_s)f(b_{(sw)d})
=f(b_{(sw)d})+(q-q^{-1})f(b_{wd})\), as required.

Now suppose that \(s\in\WD_L(\mathscr{I}_d,w)\). This means that \(sw\notin D_{L}\),
whereas \(w\in D_{L}\). So \(l(sw)>l(w)\) and \(sw=ws'\) for some \(s'\in L\), by
Lemma~\ref{deo1}. Since the definition of \(L\) gives \(s'd=dr\) for some~\(r\in J\)
we see that \(wd\in\mathscr{I}\subseteq D_J\) but \(s(wd)=(wd)r\notin D_J\).
So \(s\in\WD(\mathscr{I}\!,\,wd)\), giving \(T_sb_{wd}=-q^{-1}b_{wd}\),
and applying \(f\) to this gives \(T_sf(b_{wd})=-q^{-1}f(b_{wd})\), as required.

Finally, suppose that \(s\in\WA_L(\mathscr{I}_d,w)\), so that
\(sw\in D_L\setminus\mathscr{I}_d\). Since \(sw\in W_K\) it follows that
\(swd\in (W_K\cap D_L)d = D_{K\cap dJd^{-1}}^Kd\subseteq D_J\), by Lemma~\ref{mackey},
since \(d\in D_{K,J}\). Furthermore, since \(sw\in W_K\) and \(sw\notin\mathscr{I}_d\)
it follows that \(swd\notin\mathscr{I}\). So \(s\in\WA_J(\mathscr{I}\!,\,wd)\),
and therefore
\begin{equation}\label{Tsbwd}
T_sb_{wd}=qb_{wd}-\sum_{y<swd}r^s_{y,wd}b_y
\qquad\text{for some \(r^s_{y,wd}\in q\mathcal{A}^+\).}
\end{equation}
Since \(T_sb_{wd}\in\mathscr{S}_d\), if \(b_y\) has nonzero coefficient
in the right hand side of Eq.~\eqref{Tsbwd} then \(y\in\mathscr{J}_d\). But
\(f(b_y)=0\) if \(y\in\mathscr{J}_d'=\mathscr{J}_d\setminus \mathscr{I}_dd\).
So applying \(f\) to Eq.~\eqref{Tsbwd} gives
\[
T_sf(b_{wd})=qf(b_{wd})-\sum_yr^s_{yd,wd}f(b_y)
\]
where the sum is over elements \(y\in\mathscr{I}_d\) such that
and \(yd<swd\). Since \(l(yd)=l(y)+l(d)\) and \(l(swd)=l(sw)+l(d)\)
it follows that \(yd<swd\) if and only if \(y<sw\) (by Lemma~\ref{cancellation}).
So
\[
T_sf(b_{wd})=qf(b_{wd})-\!\!\!\sum_{y\in\mathscr{I}_d,\,y<sw}\!\!\!r^s_{yd,wd}f(b_y)
\]
which is of the required form.\qed
\end{proof}

\begin{corollary}\label{restrictedWGdetelement}
Let \(J\) and \(K\) be subsets of \(S\) and suppose that \(w\in W\) is a \(W\!\)-graph
determining element associated with~\(J\). If \(w=vd\) with \(v\in W_K\) and
\(d\in D_K^{-1}\) then \(v\) is a \(W_K\)-graph determining element associated with
\(K\cap dJd^{-1}\).
\end{corollary}

\begin{proof}
Let \(\mathscr{I}=\{\,x\in W\mid x\leqslant\lside w\,\}\), so that \((\mathscr{I}\!,\,J)\) is a
\(W\!\)-graph ideal. Clearly \(d\in D_K^{-1}\cap \mathscr{I}\!\), since \(d\leqslant\lside w\),
and it follows from Theorem~\ref{restrictedWGideal} that \((\mathscr{I}_d,K\cap dJd^{-1})\)
is a  \(W_K\)-graph ideal, where \(\mathscr{I}_d=\{\,y\in W_K\mid yd\leqslant\lside w\,\}\).
But \(yd\leqslant\lside vd\) if and only if \(y\leqslant\lside v\), since \(y,\,v\in W_K\) and
\(d\in D_K^{-1}\). So
\(\mathscr{I}_d=\{\,y\in W_K\mid y\leqslant\lside v\,\}\), and the result follows. 
\qed
\end{proof}
\begin{remark}\label{doublecosetgraphs}
Let \(\Gamma=(C,\mu,\tau)=\Gamma(\mathscr{I}\!,\,J)\), the \(W\!\)-graph obtained
from the \(W\!\)-graph ideal \((\mathscr I\!,\,J)\), and let \(K\subseteq S\). By
Eq.~\eqref{mackeyRestriction} the vertex set \(C=\{\,c_w\mid w\in\mathscr{I}\,\}\) is
expressible as a disjoint union \(\bigsqcup_d C_d\), where
\(C_d=\{\,c_{wd}\mid w\in\mathscr I_d\,\}\) and \(d\) runs through
\(D_K^{-1}\cap \mathscr{I}\). Let \(\tau_K\colon C\to\mathcal P(K)\) be defined by
\(\tau_K(c)=\tau(c)\cap K\) for all \(c\in C\), so that \(\Delta=(C,\mu,\tau_K)\) is a
\(W_K\)-graph, with \(M_\Delta\) isomorphic to the restriction of \(M_\Gamma\) to
\(\mathcal H_K\). For each \(d\in D_K^{-1}\cap\mathscr{I}\) let \(\Delta_d\) be
the full subgraph of \(\Delta\) spanned by~\(C_d\). It is clear from the results
in this section that \(\Delta_d\) is a union of cells of \(\Delta\), and
spans \(W_K\)-graph isomorphic to \(\Gamma(\mathscr{I}_d,K\cap dJd^{-1})\).
\end{remark}
In particular, it follows from Remark~\ref{doublecosetgraphs} that if
\(V\) is a closed subset of \(C\) (so that \(V\) spans an \(\mathcal H\)-submodule
of \(M_\Gamma\)) then \(V\cap C_d\) is a closed subset of~\(C_d\). Hence we obtain
the following result, which is, in a sense, dual to Theorem~\ref{indwgsubideal}.

\begin{theorem}
Let \((\mathscr L,J)\) be a strong \(W\!\)-graph subideal of the \(W\!\)-graph ideal
\((\mathscr I,J)\), and let \(K\subseteq S\). For each \(d\in D_K^{-1}\cap \mathscr{L}\)
let \(\mathscr L_d=\{\,w\in W_K\mid wd\in\mathscr{L}\,\}\) and
\(\mathscr I_d=\{\,w\in W_K\mid wd\in\mathscr{I}\,\}\). Then \((\mathscr L_d,K\cap dJd^{-1})\)
is a strong \(W_K\)-graph subideal of \((\mathscr I_d,K\cap dJd^{-1})\).
\end{theorem}

\begin{proof} Definition~\ref{strongwgsubideal} and Theorem~\ref{wgdetset} show that
\(\Gamma(\mathscr{L},J)\) can be identified with the full subgraph of
\(\Gamma(\mathscr{I}\!,\,J)\) spanned by \(\{\,c_w\mid w\in\mathscr{L}\,\}\),
and that \(V=\{\,c_w\mid w\in\mathscr{I}\setminus\mathscr{L}\,\}\) is a closed
subset of~\(C\). Hence \(V\cap C_d\) is a closed subset of \(C_d\).
Since \(V\cap C_d=\{\,c_{wd}\mid w\in\mathscr{I}_d\setminus\mathscr{L}_d\,\}\), the
result follows immediately from Definition~\ref{strongwgsubideal} and Theorem~\ref{wgdetset}.
\qed

\end{proof}

\section{\textit{W-}graph ideals for Coxeter groups of rank~2}
\label{sec:ext}

Our main aim in this section is to determine all \(W\!\)-graph ideals for
finite Coxeter groups of rank~2. Accordingly, we assume henceforth that 
\(W\) is the group generated by \(S = \{s,t\}\) subject to the defining relations
\(s^2 = t^2 = (st)^m = 1\), where \(m\geqslant 2\).

\begin{notation}
Whenever \(x\) and \(y\) are elements of a semigroup we define
\([..xy]_k\) to be \((xy)^{k/2}\) if \(k\) is even and to be \(y(xy)^{(k-1)/2}\)
if \(k\) is odd.
\end{notation}
Using this notation, \([..st]_m=[..ts]_m\) is the longest element of~\(W\),
and every other element of~\(W\) has a unique expression of the form 
\([..st]_l\) or \([..ts]_l\) with \(l<m\). Note that
\begin{align*}
D_{\{s\}} &= \{\,[..st]_l\mid l<m\,\},\\
D_{\{t\}} &= \{\,[..ts]_l\mid l<m\,\}.
\end{align*}

We assume henceforth that that \(J\subseteq S\) and that
\(\emptyset\ne\mathscr{I}\subseteq D_J\) is an ideal of \((W,\leqslant\lside)\).
Recall from \cite[Section 8]{howvan:wgraphDetSets} that \((\mathscr{I},J)\)
is a \(W\!\)-graph ideal if \(\mathscr{I} = D_J\), and note that if
\(J=\{s,t\}\) then \(D_{\{s,t\}}=\{1\}\), forcing \(\mathscr{I} = D_J\).

Suppose now that \(J=\{s\}\), and note that we must have
\[
\mathscr{I}=\{\,[..st]_l\mid l\leqslant k\,\}
\]
for some integer~\(k\) with \(0\leqslant k\leqslant m-1\). Let \(w\) be an arbitrary element
of~\(\mathscr{I}\) and let \(l(w)=l\). If \(l=0\) then \(sw=s\notin D_{\{s\}}\) and
\(w<tw=t\in\mathscr{I}\), giving \(s\in\WD(w)\) and \(t\in \SA(w)\).  If \(0<l<k\) then
\(\{sw,tw\}=\{[..st]_{l-1},[..st]_{l+1}\}\subset\mathscr{I}\); so \(s\in\SD(w)\)
and \(t\in\SA(w)\) if \(l\) is even, \(s\in\SA(w)\) and \(t\in\SD(w)\) if \(l\) is odd.
If \(l=k<m-1\) the same conclusion holds with \(\SA(w)\) replaced by \(\WA(w)\), since in
this case \([..st]_{l+1}\in D_{\{s\}}\setminus\mathscr{I}\!\). If \(l=k=m-1\), which means
that \(\mathscr{I}=D_{\{s\}}\), then \(s\in\SD(w)\) and \(t\in\WD(w)\) if \(l\) is
even, vice versa if \(l\) is odd.

It is now relatively straightforward to use (A3) and (A4) of Section~\ref{sec:5} to
compute the polynomials~\(q_{y,z}\) for \((\mathscr{I}\!,\,J)=(\mathscr{I}\!,\,\{s\})\).

\begin{lemma}\label{qyz for J=s}
With \(\mathscr{I}\) and \(J\) as above, suppose that \(y,\,z\in\mathscr{I}\) with
\(l(y)<l(z)\). Then
\[q_{y,z}=
\begin{cases} 1&\text{if \(l(z)-l(y)=1\),}\\
0&\text{if \(l(z)-l(y)>1\)}.
\end{cases}
\]
\end{lemma}
\begin{proof}
The proof proceeds by induction on \(l(z)\). If \(l(z)=1\) then \(z=t\) and \(y=1\),
and~(A4) immediately gives \(q_{y,z}=1\), as required. 

For the inductive step, suppose first that \(l(z)\) is even. Then \(s\in\D(z)\),
and \(sz\) is the only element of~\(\mathscr{I}\) whose length is~\(l(z)-1\). Since
(A4) immediately gives \(q_{sz,z}=1\), it suffices to prove that \(q_{y,z}=0\) if
\(l(y)<l(z)-1\).

If \(l(y)\) is odd then \(s\in\A(y)\), and \(l(y)<l(z)-1\) gives \(l(y)\leqslant l(z)-3<l(sz)-1\).
So the inductive hypothesis gives \(q_{y,sz}=0\), and by (1) of (A4) it follows that 
\(q_{y,z}=qq_{y,sz}=0\).

Assume now that \(l(y)\) is even, so that \(s\in \D(y)\). Since \(l(y)\leqslant l(z)-2<l(sz)\)
the inductive hypothesis tells us that \(q_{y,sz}\) is a constant, and so
\(q^{-1}(q_{y,sz}-\mu_{y,sz})=0\). If \(s\in\SD(y)\) then
\(l(sy)=l(y)-1<l(z)-1=l(sz)\), and the inductive hypothesis gives \(q_{sy,sz}=0\).
So whether \(s\in\SD(w)\) or \(s\in\SA(w)\) we have
\(q_{y,z} = \sum_x\mu_{y,x}q_{x,sz}\), where the sum extends over \(x\in \mathscr I\)
such that \(y<x<sz\) and \(s\notin\D(x)\). But \(s\notin\D(x)\) implies that \(l(x)\)
is even, giving \(l(x)<l(sz)-1\), since \(l(sz)\) is also even. Since this gives
\(q_{x,sz}=0\) by the inductive hypothesis it follows that all the terms in the
sum are~0, and \(q_{y,z}=0\), as required.

If \(l(z)\) odd then the same proof applies, with odd and even swapped and
with \(s\) replaced by~\(t\). This completes the induction.
\qed
\end{proof}
It follows from Lemma~\ref{qyz for J=s} and the discussion preceding it that if
\(k<m-1\) then the construction produces a graph of the form
\[
\begin{tikzpicture}[scale=0.30][font=\small]
\draw (0,0) circle [radius=0.6];
\node at (0,0) {\(s\)};
\draw [<->] (0.7,0) --(2.3,0);
\draw (3,0) circle [radius=0.6];
\node at (3,0) {\(t\)};
\draw [<->] (3.7,0) --(5.3,0);
\draw (6,0) circle [radius=0.6];
\node at (6,0) {\(s\)};
\draw [<->] (6.7,0) --(8.3,0);
\draw (9,0) circle [radius=0.6];
\node at (9,0) {\(t\)};
\draw [<->] (9.7,0) --(11.3,0);
\node at (12,0) {\(\cdots\)};
\end{tikzpicture}
\]
where the number of vertices is \(k+1\) and all edges have weight~1. In other
words, if we let \(V=\{v_1,v_2, \ldots, v_{k+1}\}\) be the vertex set, then 
\(\tau\colon V\to\powerset(S)\) is given by
\[
\tau(v_i)=\begin{cases}
\{s\}&\text{if \(i\) is odd,}\\
\{t\}&\text{if \(i\) is even,}
\end{cases}
\]
and the integer \(\mu(v_i,v_j)\)
is~\(1\) whenever \(|i-j|=1\) and is \(0\) whenever \(|i-j|>1\). It follows from
Theorem~\ref{converseMain} that \((\mathscr{I}\!,\,J)\) is a \(W\!\)-graph ideal
if and only if \(\Gamma=(V,\mu,\tau)\) is a \(W\!\)-graph.

Note that if \(k=m-2\) then \(\mathscr{I}=D_J\setminus\{[..st]_{m-1}\}\). In this
case it follows from results already obtained \((\mathscr{I}\!,\,J)\) is a
\(W\!\)-graph ideal. Indeed, we saw in Section~\ref{sec:4} that \((D_J,J)\) is a \(W\!\)-graph
ideal, and since \(\D([..st]_{m-1})=\{s,t\}\) (as noted in the discussion above), it
follows that the set \(\{[..st]_{m-1}\}\) is \((D_J,J)\)-closed. Hence
\((D_J\setminus\{[..st]_{m-1}\},J)\)  is a strong \(W\!\)-graph subideal of~\((D_J,J)\).

The next lemma shows that \((V,\mu,\tau)\) is a \(W\!\)-graph if and only if \(k+2\)
is a divisor of~\(m\).

\begin{lemma}\label{irreducibleRepI2m}
Let \(M\) be a free \(\mathcal{A}\)-module with \(\mathcal{A}\)-basis
\(V = \{v_1, \ldots, v_{k+1}\}\), where \(k\geqslant 0\), and for each \(r\in\{s,t\}\)
let \(\phi_r\colon M\to M\) be the \(\mathcal{A}\)-homomorphism
satisfying
\[
\phi_r(v_i)=\begin{cases}
-q^{-1}v_i&\text{ if \(\tau(v_i)=\{r\}\)}\\
qv_i+\sum\limits_{j\in\mathcal{R}_i}v_j&\text{ if \(\tau(v_i)\ne\{r\}\)}
\end{cases}
\]
where \(\mathcal{R}_i=\{i-1,i+1\}\cap\{1,2,\ldots,k+1\}\). Then the relation
\(\phi_r^2=1+(q-q^{-1})\phi_r\) is satisfied for both values of \(r\in\{s,t\}\),
and \([..\phi_s\phi_t]_n=[..\phi_t\phi_s]_n\) if and only if \(n\)
is a multiple of \(k+2\).
\end{lemma}

\begin{proof}
Observe that if \(\tau(v_i)\ne\{r\}\) then \(\tau(v_j)=\{r\}\) for all
\(j\in\mathcal{R}_i\). It follows by a trivial calculation that
\(\phi_r^2=1+(q-q^{-1})\phi_r\).

If \(m=k+2\) then \(M\) is isomorphic to the \(\mathcal{H}\)-module \(M_\Gamma\),
where \(\Gamma=\Gamma(D_J\setminus\{[..st]_{m-1}\},J)\), with
\(T_s\) acting via \(\phi_s\) and \(T_t\) acting via \(\phi_t\).
Hence \([..\phi_s\phi_t]_{k+2}=[..\phi_t\phi_s]_{k+2}\). It follows
from this that also \([..\phi_s\phi_t]_n=[..\phi_t\phi_s]_n\) whenever
\(n\) is a multiple of~\(k+2\). It remains to prove the converse:
if \([..\phi_s\phi_t]_n=[..\phi_t\phi_s]_n\) then \(n\)
is a multiple of~\(k+2\).

So assume that \([..\phi_s\phi_t]_n=[..\phi_t\phi_s]_n\). If \(k=0\) then
\(\phi_s(v_1)=-q^{-1}v_1\) and \(\phi_t(v_1)=qv_1\), and it follows that
if \(n=2l+1\) is odd then
\([..\phi_s\phi_t]_n=(-1)^l\phi_s\ne(-1)^l\phi_t=[..\phi_t\phi_s]_n\),
contrary to our hypothesis. So \(n\) is even, as required.

Assume now that \(k\geqslant 1\). It is convenient to regard \(M\) as embedded in a
\(\mathbb{C}[q,q^{-1}]\)-module with basis~\(V\), and extend \(\phi_s\) and
\(\phi_t\) to \(\mathbb{C}[q,q^{-1}]\)-endomorphisms of this module. 
Let \(\zeta\) be a primitive \(2(k+2)\)-th root of unity, and write
\(\theta_k=\zeta^k-\zeta^{-k}\) for all integers~\(k\).

Define
\(u_1 = \sum_{i\in O}\theta_i v_i\) and \(u_2 = \sum_{i\in E}\theta_i v_i\),
where \(O\) and \(E\) are respectively the set of odd integers in \(\{1,2,\ldots,k+1\}\)
and the set of even integers in \(\{1,2,\ldots,k+1\}\). It is easily seen that
\(\phi_s(u_1) = -q^{-1}u_1\) and \(\phi_t(u_2) = -q^{-1}u_2\), while
\begin{align*}
\phi_s(u_2)&=qu_2 + \sum_{i\in O}(\theta_{i+1}-\theta_{i-1})v_i\\
\phi_t(u_1)&=qu_1 + \sum_{i\in E}(\theta_{i+1}-\theta_{i-1})v_i
\end{align*}
since \(\theta_0=\theta_{k+2}=0\). Now since
\(\theta_{i+1}-\theta_{i-1}=(\zeta+\zeta^{-1})\theta_i\) it follows that
the two-dimensional submodule spanned by \(\{u_1,u_2\}\) is preserved by
both \(\phi_s\) and \(\phi_t\), which act via the the following two
matrices:
\[
F_s=\begin{pmatrix}
-q^{-1}&\ \zeta+\zeta^{-1}\\
0&q\
\end{pmatrix},
\qquad
F_t=\begin{pmatrix}
q\ \ &0\\
\zeta+\zeta^{-1}&\ -q^{-1}
\end{pmatrix}.
\]
Since \([..\phi_s\phi_t]_n=[..\phi_t\phi_s]_n\) it follows that
\([..F_sF_t]_n=[..F_tF_s]_n\). This must remain
valid on specializing to~\(q=1\), in which case \(F_s^2=F_t^2=1\)
and \((F_sF_t)^n=([..F_tF_s]_n)^{-1}[..F_sF_t]_n=1\).
But since
\[
F_sF_t=\begin{pmatrix}
\zeta^2+\zeta^{-2}+1&\ -(\zeta+\zeta^{-1})q^{-1}\\
(\zeta+\zeta^{-1})q&-1
\end{pmatrix}
\]
and the eigenvalues of this are \(\zeta^2\) and \(\zeta^{-2}\), it follows
that \((\zeta^2)^n=1\). Since \(\zeta^2\) is a primitive \((k+2)\)-th root
of~1 we conclude that \(k+2\) is a divisor of~\(n\), as required.\qed
\end{proof}
Suppose now that \(J=\emptyset\), so that \(D_J=W\). Since we know that
\((W,\emptyset)\) is a \(W\!\)-graph ideal, we assume that \(\mathscr{I}\)
is an ideal of \((W,\leqslant\lside)\) such that \(\mathscr{I}\ne W\). Then
\[
\mathscr{I}=\mathscr{I}_{h,k}=\{\,[..st]_l\mid l\leqslant h\,\}
\cup\{\,[..ts]_l\mid l\leqslant k\,\}
\]
for some \(h,\,k\in\{0,1,2,\ldots,m-1\}\). Since \(D_J=W\) there are no weak
descents. So \(\D(1)=\emptyset\), and for every other \(w\in\mathscr{I}\)
we have either \(\D(w)=\{s\}\) (if the reduced expression for~\(w\) starts
with~\(s\)) or \(\D(w)=\{t\}\) (if it starts with~\(t\)). 

For the purposes of applying Theorem~\ref{converseMain} we need to find
the integers \(\mu_{y,w}\) that appear in (A2) of Section~\ref{sec:5}. This
means that \(y<w\) and \(\D(y)\nsubseteq\D(w)\). Clearly we may as well assume 
that \(\D(w)=\{s\}\) and \(\D(y)=\{t\}\).

\begin{lemma}\label{mu for Ihk}
Let \(\mathscr{I}=\mathscr{I}_{k,h}\) (as defined above) and let
\(J=\emptyset\). Let \(y,\,w\) be elements of \(\mathscr{I}\) with
\(\D_J(\mathscr{I}\!,\,w)=\{s\}\)
and \(\D_J(\mathscr{I}\!,\,y)=\{t\}\), and \(0<l(y)<l(w)\). Then 
\(\mu_{y,w}=1\) if \(l(w)-l(y)=1\), and \(\mu_{y,w}=0\) otherwise.
\end{lemma}

\begin{proof}
If \(l(w)-l(y)=1\) then \(y=sw\), and it is immediate from (A4) of
Section~\ref{sec:5} that \(\mu_{y,w}=q_{y,w}=1\). If \(l(w)-l(y)>1\)
then case~(1) of (A4) applies, since \(s\in\A(y)\), and so
\(q_{y,w}=qq_{y,sw}\). So the constant term of \(q_{y,w}\) is
zero, as required.\qed
\end{proof}
So, after removing superfluous edges, the graph produced by application
of our algorithm to \((\mathscr{I}_{h,k},\emptyset)\) has the form
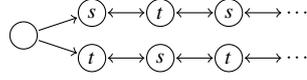
\[
\begin{tikzpicture}[scale=0.3][font=\small]
\draw (0,0) circle [radius=0.6];
\node at (0,0) {\(t\)};
\draw [<->] (0.7,0) --(2.3,0);
\draw (3,0) circle [radius=0.6];
\node at (3,0) {\(s\)};
\draw [<->] (3.7,0) --(5.3,0);
\draw (6,0) circle [radius=0.6];
\node at (6,0) {\(t\)};
\draw [<->] (6.7,0) --(8.3,0);
\node at (9,0) {\(\cdots\)};
\draw (0,2) circle [radius=0.6];
\node at (0,2) {\(s\)};
\draw [<->] (0.7,2) --(2.3,2);
\draw (3,2) circle [radius=0.6];
\node at (3,2) {\(t\)};
\draw [<->] (3.7,2) --(5.3,2);
\draw (6,2) circle [radius=0.6];
\node at (6,2) {\(s\)};
\draw [<->] (6.7,2) --(8.3,2);
\node at (9,2) {\(\cdots\)};
\draw(-3,1) circle [radius=0.6];
\draw [->] (-2.336,0.779) --(-0.664,0.221);
\draw [->] (-2.336,1.221) --(-0.664,1.779);
\end{tikzpicture}
\]
where there are \(h+k+1\) vertices, \(k\) in the
top row and \(h\) in the bottom row, and all edges have weight~1.
In other words, if we let
\(V=\{\,v_i\mid 1\leqslant i\leqslant k\,\}\cup\{x\}\cup\{\,u_i\mid 1\leqslant i\leqslant h\,\}\)
be the vertex set, where the \(v_i\) correspond to the top row and
the \(u_i\) to the bottom row, and temporarily 
let \(v_0=x\) and \(v_{-i}=u_i\) for \(1\leqslant i\leqslant h\), then \(\tau\colon V\to\powerset(S)\)
is given by \(\tau(v_0)=\emptyset\) and
\[
\tau(v_i)=\begin{cases}
\{s\}&\text{if \(i\) is odd and positive or even and negative,}\\
\{t\}&\text{if \(i\) is even and positive or odd and negative,}
\end{cases}
\]
and the integer \(\mu(v_i,v_j)\)
is~\(1\) whenever \(|i-j|=1\) and is \(0\) whenever \(|i-j|>1\). It follows from
Theorem~\ref{converseMain} that \((\mathscr{I}_{h,k},\emptyset)\) is a
\(W\!\)-graph ideal if and only if \(\Gamma=(V,\mu,\tau)\) is a~\(W\!\)-graph.

Note that in the particular case \(h=0\) and \(k=m-1\) we have
\(\mathscr{I}_{h,k}=D_{\{s\}}\), and it
follows from \cite[Proposition~8.3]{howvan:wgraphDetSets} that
\((\mathscr{I}_{h,k},\emptyset)\) is a \(W\!\)-graph ideal.

Our next lemma shows that, in the general case, \((\mathscr{I}_{h,k},\emptyset)\)
is a \(W\!\)-graph ideal if and only if \(h+1\) and \(k+1\) are both
divisors of~\(m\).

\begin{lemma}
Let \(M\) be a free \(\mathcal{A}\)-module with \(\mathcal{A}\)-basis
\(\{x\}\sqcup\{u_1,u_2,\ldots,u_h\}\sqcup\{v_1,v_2,\ldots,v_k\}\), and
put \(u_0=v_0=u_{h+1}=v_{k+1}=0\). Let \(\phi_s\) and \(\phi_t\) be
\(\mathcal{A}\)-endomorphisms  of \(M\) satisfying the following rules:
\begin{itemize}[topsep=1 pt]
\item[\textup{(i)}]\(\phi_s(x)=qx+v_1\) and \(\phi_t(x)=qx+u_1\),
\item[\textup{(ii)}]\(\phi_s(v_i)=-q^{-1}v_i\) if \(i\) is odd, and
\(\phi_s(v_i)=qv_i+v_{i-1}+v_{i+1}\) if \(i\) is even,
\item[\textup{(iii)}]\(\phi_s(u_i)=-q^{-1}u_i\) if \(i\) is even, and
\(\phi_s(u_i)=qu_i+u_{i-1}+u_{i+1}\) if \(i\) is odd,
\item[\textup{(iv)}]\(\phi_t(v_i)=-q^{-1}v_i\) if \(i\) is even, and
\(\phi_t(v_i)=qv_i+v_{i-1}+v_{i+1}\) if \(i\) is odd,
\item[\textup{(v)}]\(\phi_t(u_i)=-q^{-1}u_i\) if \(i\) is odd, and
\(\phi_t(u_i)=qu_i+u_{i-1}+u_{i+1}\) if \(i\) is even.
\end{itemize}
Then \(\phi_s^2=1+(q-q^{-1})\phi_s\) and \(\phi_t^2=1+(q-q^{-1})\phi_t\),
and \([..\phi_s\phi_t]_n=[..\phi_t\phi_s]_n\) if and only if \(h+1\)
and \(k+1\) are both divisors of~\(n\).
\end{lemma}

\begin{proof}
Checking that \(\phi_s^2=1+(q-q^{-1})\phi_s\)
and \(\phi_t^2=1+(q-q^{-1})\phi_t\) is straightforward.

If \(h=0\) and \(m=k+1\) then \(M\) is isomorphic to the \(\mathcal{H}\)-module
\(M_\Gamma\), where \(\Gamma=\Gamma(D_{\{s\}},\emptyset)\), with
\(T_s\) acting via \(\phi_s\) and \(T_t\) acting via \(\phi_t\).
Hence \([..\phi_s\phi_t]_{k+1}=[..\phi_t\phi_s]_{k+1}\) if \(h=0\). It follows
from this that also \([..\phi_s\phi_t]_n=[..\phi_t\phi_s]_n\) whenever \(h=0\)
and \(n\) is a multiple of~\(k+1\). Similarly, \([..\phi_s\phi_t]_n=[..\phi_t\phi_s]_n\)
whenever \(k=0\) and \(n\) is a multiple of~\(h+1\). 

Turning to the general case, let \(M_U\) be the \(\mathcal{A}\)-submodule
of \(M\) spanned by \(\{u_1,u_2,\ldots,u_h\}\) and let \(M_V\) be the
\(\mathcal{A}\)-submodule of \(M\) spanned by \(\{v_1,v_2,\ldots,v_k\}\). Note
that \(M_U\) and \(M_V\) are both invariant under \(\phi_s\) and \(\phi_t\).
Let \(G_s\) and \(G_t\) be the matrices of \(\phi_s\) and \(\phi_t\) on~\(M_U\),
relative to the ordered basis \((u_h,u_{h-1},\ldots,u_1)\), and let \(F_s\) and
\(F_t\) be the matrices of \(\phi_s\) and \(\phi_t\) on~\(M_V\),
relative to the ordered basis \((v_1,v_2,\ldots,v_k)\). Then 
the matrices of \(\phi_s\) and \(\phi_t\) on \(M\) relative to the
ordered basis \((u_h,u_{h-1},\ldots,u_1,x,v_1,\ldots,v_{k-1},v_k)\) are
\[
H_s=\begin{bmatrix}
G_s&0&0\\
0&q&0\\
0&v&F_s
\end{bmatrix}
\text{\qquad and\qquad}
H_t=\begin{bmatrix}
G_t&u&0\\
0&q&0\\
0&0&F_t
\end{bmatrix}
\]
where all entries of the columns \(u\) and \(v\) are zero, except for the last entry
of \(u\) and the first entry of~\(v\), which are both~1. 

If \([..\phi_s\phi_t]_n=[..\phi_t\phi_s]_n\) then \([..G_sG_t]_n=[..G_tG_s]_n\), and
it follows by Lemma~\ref{irreducibleRepI2m} that \(h+1\) must be a divisor of~\(n\).
Similarly also \([..F_sF_t]_n=[..F_tF_s]_n\), and it follows by
Lemma~\ref{irreducibleRepI2m} that \(k+1\) must be a divisor of~\(n\). It remains to
prove that if \(h+1\) and \(k+1\) are divisors of~\(n\) then \([..H_sH_t]_n=[..H_tH_s]_n\).

Assume that \(h+1\) and \(k+1\) are divisors of~\(n\). Observe that \(\phi_s\) and
\(\phi_t\) act on the quotient module \(M/M_U\) via the following two matrices,
\[
H_s'=\begin{bmatrix}
q&0\\
v&F_s
\end{bmatrix}
\text{\qquad and\qquad}
H_t'=\begin{bmatrix}
q&0\\
0&F_t
\end{bmatrix}
\]
which are also the matrices of \(\phi_s\) and \(\phi_t\) on \(M\) in the case~\(h=0\). 
Since \([..\phi_s\phi_t]_n=[..\phi_t\phi_s]_n\) in this case, it follows that
\([..H_s'H_t']_n=[..H_t'H_s']_n\). Similarly the matrices
\[
H_s''=\begin{bmatrix}
G_s&0\\
0&q
\end{bmatrix}
\text{\qquad and\qquad}
H_t''=\begin{bmatrix}
G_t&u\\
0&q
\end{bmatrix}
\]
satisfy \([..H_s''H_t'']_n=[..H_t''H_s'']_n\). But it is clear that
\[
[..H_sH_t]_n=\begin{bmatrix}
[..G_sG_t]_n&*&0\\
0&q^n&0\\
0&*&[..F_sF_t]_n
\end{bmatrix}
=
\begin{bmatrix}
[..H_s''H_t'']_n&0\\ *&[..F_sF_t]_n
\end{bmatrix}
=\begin{bmatrix}
[..G_sG_t]_n&*\\ 0&[..H_s'H_t']_n
\end{bmatrix}
\]
where the asterisks mark entries whose values are irrelevant to our argument.
Moreover
\[
[..H_tH_s]_n=\begin{bmatrix}
[..G_tG_s]_n&*&0\\
0&q^n&0\\
0&*&[..F_tF_s]_n
\end{bmatrix}
=
\begin{bmatrix}
[..H_t''H_s'']_n&0\\ *&[..F_tF_s]_n
\end{bmatrix}
=\begin{bmatrix}
[..G_tG_s]_n&*\\ 0&[..H_t'H_s']_n
\end{bmatrix}
\]
by similar calculations, and since \([..H_s'H_t']_n=[..H_t'H_s']_n\) and
 \([..H_s''H_t'']_n=[..H_t''H_s'']_n\) it follows that
 \([..H_sH_t]_n=[..H_tH_s]_n\), as required.
\qed
\end{proof}
The following theorem gathers together the various results proved above,
and their obvious analogues obtained by swapping \(s\)~and~\(t\).

\begin{theorem}\label{dihedral}
Let \((W,S)\) be a Coxeter system of type \(I_2(m)\), and let \(S=\{s,t\}\).
Then \((\mathscr{I}\!,\,J)\) is a \(W\!\)-graph ideal if and only if
one of the following alternatives is satisfied:
\begin{itemize}[topsep=1 pt]
\item[\textup{(i)}]\((\mathscr{I}\!,\,J)=(\{1\},S)\),
\item[\textup{(ii)}]\((\mathscr{I}\!,\,J)=(D_{\{s\}},\{s\})\),
\item[\textup{(iii)}]\((\mathscr{I}\!,\,J)=(\{\,[..st]_l\mid l\leqslant k\,\},\{s\})\), where \(k+2\)
divides~\(m\),
\item[\textup{(iv)}]\((\mathscr{I}\!,\,J)=(D_{\{t\}},\{t\})\),
\item[\textup{(v)}]\((\mathscr{I}\!,\,J)=(\{\,[..ts]_l\mid l\leqslant k\,\},\{t\})\), where \(k+2\)
divides~\(m\),
\item[\textup{(vi)}]\((\mathscr{I}\!,\,J)=(W,\emptyset)\),
\item[\textup{(vii)}]\((\mathscr{I}\!,\,J)=(\{\,[..st]_l\mid l\leqslant h\,\}
\cup\{\,[..ts]_l\mid l\leqslant k\,\},\emptyset)\), where \(h+1\) and \(k+1\) divide~\(m\).
\end{itemize}
\end{theorem}
Our final objective is to determine all the \(W\!\)-graph biideals in type \(I_2(m)\).
We need the following lemma.

\begin{lemma}\label{Twc1}
With \((W,S)\) as above, let \(\mathscr{I}=\{\,[..st]_l\mid l\leqslant h\,\}
\cup\{\,[..ts]_l\mid l\leqslant k\,\}\), where \(h\) and \(k\) are nonnegative
integers, and assume that \((\mathscr{I}\!,\,\emptyset)\) is a \(W\!\)-graph ideal.
Let \(C=\{\,c_w\mid w\in\mathscr{I}\,\}\) be the \(W\!\)-graph basis of the
\(\mathcal{H}\)-module \(\mathscr{S}(\mathscr{I}\!,\,\emptyset)\), and let
\(w\in \mathscr{I}\!\) with \(l(w)\leqslant \min(h,k)+1\). Then
\(T_wc_1=c_w+\sum_xq^{l(w)-l(x)}c_x\), where \(x\) runs through the set
\(\{\,x\in W\mid l(x)<l(w)\,\}\).
\end{lemma}

\begin{proof}
Note first that \(\mathscr{I}\) contains all elements of \(W\) such that
\(l(w)\leqslant\min(h,k)\), and hence contains all \(x\) such that \(l(x)<l(w)\).

We use induction on \(l(w)\). If \(l(w)=0\) the statement becomes \(T_1c_1=c_1\),
which is true since \(T_1\) is the identity element of~\(\mathcal{H}\). So
assume that \(l(w)=l>0\), and let \(w=rv\) with \(r\in\{s,t\}\) and \(l(v)=l-1\).
Since the proofs for the two cases are essentially the same, we shall only do the
case \(r=s\).

Recall that the edge weights for \(\Gamma(\mathscr{I}\!,\,\emptyset)\)
were found in Lemma~\ref{mu for Ihk}. This makes it easy to evaluate \(T_sc_x\)
for all~\(x\in\mathscr{I}\). In particular, \(T_sc_1=qc_1+c_s\). This shows that
the desired formula holds when \(w=sv\) and \(v=1\). So henceforth we assume that
\(v\ne 1\). Note that since \(l(sv)>l(v)\) it follows that \(l(tv)<l(v)\).

Observe that \(\{v\}\cup\{\,x\in W\mid l(x)<l(v)\,\}\) is a union of right cosets
of the group \(\{1,t\}\), namely those cosets whose minimal element has length
\(l-2\) or less. So the inductive hypothesis can be written as
\[
T_vc_1=\sum_{x\in\mathcal{E}}q^{l(v)-l(x)-1}(qc_{x}+c_{tx}),
\]
where \(\mathcal{E}=\{\,x\in D_{\{t\}}^{-1}\mid l(x)\leqslant l-2\,\}\). Similarly,
the set \(\{w\}\cup\{\,x\in W\mid l(x)<l(w)\,\}\) is a union of right cosets
of \(\{1,s\}\). Writing \(\mathcal{F}=\{\,x\in D_{\{s\}}^{-1}\mid l(x)\leqslant l-1\,\}\),
our aim is to show that 
\[
\postdisplaypenalty=10000
T_wc_1=\sum_{x\in\mathcal{F}}q^{l(w)-l(x)-1}(qc_{x}+c_{sx}).
\]
Observe that \(\{\,tx\mid x\in\mathcal{E}\,\}=\mathcal{F}\setminus\{1\}\).

If \(x\in \mathcal{E}\) and \(x\ne 1\) then \(\D(x)=\{s\}\) and \(\D(tx)=\{t\}\). Note
also that \(stx\in\mathscr{I}\), since either \(l(stx)<l(w)\) or \(stx=w\). So
\begin{align*}
T_s(qc_x+c_{tx})&=-c_x+(qc_{tx}+c_{stx}+c_x)\\
&=qc_{tx}+c_{stx}.
\end{align*}
When \(x=1\) we get \(T_s(qc_x+c_{tx})=T_s(qc_1+c_t)=q^2c_1+qc_s+qc_t+c_{st}\). So
\begin{align*}
T_wc_1=T_s(T_vc_1)&=q^{l(v)-1}(q^2c_1+qc_s+qc_t+c_{st})
+\sum_{x\in\mathcal{E}\setminus\{1\}}q^{l(v)-l(x)-1}(qc_{tx}+c_{stx})\\
&=q^{l(w)-1}(qc_1+c_s)+q^{l(w)-2}(qc_t+c_{st})
+\sum_{y\in\mathcal{F}\setminus\{1,t\}}q^{l(v)-l(y)}(qc_{y}+c_{sy})\\
&=\sum_{y\in\mathcal{F}}q^{l(w)-l(y)-1}(qc_{y}+c_{sy})
\end{align*}
as required.
\qed
\end{proof}

\begin{proposition}\label{isbiideal}
Let \((W,S)\) be a Coxeter system of type \(I_2(m)\), with \(S=\{s,t\}\). Let
\(k\) be a nonnegative integer such that \(k+1\) divides~\(m\), and let
\(\mathscr{I}=\{\,w\in W\mid l(w)\leqslant k\,\}\). Then \((\mathscr{I},\emptyset,\emptyset)\)
is a \(W\!\)-graph biideal.
\end{proposition}

\begin{proof}
By case (vii) in Theorem~\ref{dihedral} we know that \((\mathscr{I},\emptyset)\) is a  
\(W\!\)-graph ideal, and since \(\mathscr{I}=\mathscr{I}^{-1}\) it follows that
\((\mathscr{I},\emptyset)\) is also a \(W\!\)-graph right ideal. Identifying
\(\mathscr{S}\opp(\mathscr{I}\!,\,\emptyset)\) with
\(\mathscr{S}(\mathscr{I}\!,\,\emptyset)\) by putting \(b_w\opp = b_w\)
for all~\(w\in\mathscr{I}\), the task is to show that the left and right actions
of \(\mathcal{H}\) commute.

Note that if \(k=m-1\) then \(\mathscr{I}=W\setminus\{w_S\}\), where \(w_S=[..st]_m\)
is the longest element of~\(W\). But \((W,\emptyset,\emptyset)\) is a
\(W\!\)-graph biideal, by Remark~\ref{regular biideal}, and \(\{c_{W_S}\}\) is closed
for both the left and right actions. So it follows from Theorem~\ref{subbiideal}
that \((\mathscr{I},\emptyset,\emptyset)\) is a \(W\!\)-graph biideal in this case.

Since the standard basis and \(W\!\)-graph basis of
\(\mathscr{S}(\mathscr{I}\!,\,\emptyset)\) are related by the rule that \(b_w=T_wc_1\)
for all \(w\in\mathscr{I}\), it follows from Proposition~\ref{Twc1} that
\(b_w=c_w+\sum_{v<w}q^{l(w)-l(v)}c_v\) for all \(w\in\mathscr{I}\!\). The right
ideal analogue of Proposition~\ref{Twc1} gives
\(b_w\opp=c_w\opp+\sum_{v<w}q^{l(w)-l(v)}c_v\opp\) for all \(w\in\mathscr{I}\!\).
Since \(b_w\opp=b_w\), we must have \(c_w\opp=c_w\) for all~\(w\in\mathscr{I}\!\).

The left and right actions of \(T_s\) and \(T_t\) are given by rules that are independent
of the value of~\(m\). For example, for all \(w\in\mathscr{I}\!\),
\[
T_sc_w=\begin{cases}
-q^{-1}c_w &\text{if the reduced expression for \(w\) starts with~\(s\),}\\
qc_1+c_s &\text{if \(w=1\),}\\
qc_t+c_{st} &\text{if \(w=t\),}\\
qc_w+c_{sw}+c_{tw} &\text{if the reduced expression for \(w\) starts with~\(t\) and \(1<l(w)<k\),}\\
qc_w+c_{tw} &\text{if the reduced expression for \(w\) starts with~\(t\) and \(l(w)=k\).}
\end{cases}
\]
If it happens that \(m=k+1\) then, as we have seen,
\((\mathscr{I},\emptyset,\emptyset)\) is a \(W\!\)-graph biideal, and so the
left and right actions commute. Since the value of~\(m\) is irrelevant, the left and
right actions always commute.
\qed
\end{proof}

\begin{proposition}\label{notbiideal}
Let \((W,S)\) be a Coxeter system of type \(I_2(m)\), with \(S=\{s,t\}\). Let
\(h\) and \(k\) be integers in \(\{1,2,\ldots,m-1\}\) with \(|h-k|=1\).
Let \(\mathscr{I}=\{\,[..st]_l\mid l\leqslant h\,\}
\cup\{\,[..ts]_l\mid l\leqslant k\,\}\). Then \((\mathscr{I}\!,\,\emptyset,\emptyset)\)
is not a \(W\!\)-graph biideal.
\end{proposition}
\begin{proof}
Suppose, for a contradiction, that \((\mathscr{I}\!,\,\emptyset,\emptyset)\) is
a \(W\!\)-graph biideal. It is obvious that essentially the same proof will
apply whether \(h=k-1\) or \(k=h-1\). So we assume that \(h=k-1\), which means that
\([..st]_k\) is not in \(\mathscr{I}\) and \([..ts]_k\) is in \(\mathscr{I}\).
Let \(\{\,c_w\mid w\in\mathscr{I}\,\}\) be the \((W\times W\opp)\)-graph basis
of the \((\mathcal{H},\mathcal{H})\)-bimodule
\(M=\mathscr{S}(\mathscr{I}\!,\,\emptyset,\emptyset)\).

Put \(w=[..st]_{k-1}\), and suppose first that \(k\) is even. We shall show that
\((T_sc_w)T_s\ne T_s(c_wT_s)\), contradicting the fact that \(M\) is a bimodule.
In the first instance we assume that \(k>2\), although the calculations are much
the same in the case \(k=2\). Given that \(k>2\) the reduced expression for
\(w\) starts with~\(t\) and ends with~\(t\), and there is at least one \(s\) in
between. Observe that \(c_wT_s=qc_w+c_{wt}+c_{ws}\) but
\(T_sc_w=qc_w+c_{tw}\), since \(sw\notin\mathscr{I}\). Note also that
\(ws\) is the longest element of~\(\mathscr{I}\). So we find that
\[
(T_sc_w)T_s=qc_wT_s+c_{tw}T_s=q(qc_w+c_{ws}+c_{wt})+(qc_{tw}+c_{tws}+c_{twt}),
\]
whereas
\[
T_s(c_wT_s)=qT_sc_w+T_sc_{wt}+T_sc_{ws}
=q(qc_w+c_{tw})+(qc_{wt}+c_{swt}+c_{twt})+(qc_{ws}+c_{tws}).
\]
The two expressions are not equal: the second features a \(c_{swt}\) that does
not appear in the first.
If \(k=2\) then we find that
\[
(T_sc_t)T_s=(qc_t+c_1)T_s=q(qc_t+c_{ts})+qc_1+c_s,
\]
whereas
\[
T_s(c_tT_s)=T_s(qc_t+c_1+c{ts})=q(qc_t+c_{st})+(qc_1+c_s)+(qc_{ts}+c_s),
\]
and again the two expressions are not equal.

When \(k\) is odd similar calculations show that \((T_tc_w)T_s\ne T_t(c_wT_s)\).
If \(k=3\) then
\[
(T_tc_{st})T_s=(qc_{st}+c_t)T_s=q(qc_{st}+c_{sts}+c_{s})+(qc_{t}+c_{ts})
\]
whereas
\[
T_t(c_{st}T_s)=T_t(qc_{st}+c_{sts}+c_s)=q(qc_{st}+c_t)+(qc_{sts}+c_{ts})+(qc_s+c_{ts}),
\]
and if \(k\geqslant 5\) then
\[
(T_tc_{w})T_s=(qc_{w}+c_{sw})T_s=q(qc_{w}+c_{ws}+c_{wt})+(qc_{sw}+c_{swt}+c_{sws})
\]
whereas
\[
T_t(c_{w}T_s)=T_t(qc_{w}+c_{ws}+c_{wt})=q(qc_{w}+c_{sw})+(qc_{ws}+c_{sws})+(qc_{wt}+c_{twt}+c_{swt}).
\]
A contradiction has been obtained in all cases.
\qed
\end{proof}

\begin{theorem}
Let \((W,S)\) be a Coxeter system of type \(I_2(m)\), and let \(S=\{s,t\}\).
Then \((\mathscr{I}\!,\,J,K)\) is a \(W\!\)-graph biideal if and only if
one of the following alternatives is satisfied:
\begin{itemize}[topsep=1 pt]
\item[\textup{(i)}]\((\mathscr{I}\!,\,J,K)=(W,\emptyset,\emptyset)\),
\item[\textup{(ii)}]\((\mathscr{I}\!,\,J,K)=(\{\,w\in W\mid l(w)\leqslant k\,\},\emptyset,\emptyset)\),
where \(k+1\) divides~\(m\),
\item[\textup{(iii)}]\((\mathscr{I}\!,\,J,K)=(\{1,t\},\emptyset,\emptyset)\) and \(m\) is even,
\item[\textup{(iv)}]\((\mathscr{I}\!,\,J,K)=(\{1,s\},\emptyset,\emptyset)\) and \(m\) is even,
\item[\textup{(v)}]\(\mathscr{I}=\{1\}\) and \(m\) is even, and \(J,\,K\) are any subsets of~\(S\),
\item[\textup{(vi)}]\(\mathscr{I}=\{1\}\) and \(m\) is odd, and \(J,\,K\in\{\emptyset,S\}\).
\end{itemize}
\end{theorem}
\begin{proof}
Let us first check that \((\mathscr{I}\!,\,J,K)\) is a \(W\!\)-graph biideal if it is in
the list. For case~(i) Remark~\ref{regular biideal} applies, and for case~(ii)
Proposition~\ref{isbiideal} applies. For case~(iii), observe that
\((\mathscr{I}\!,\,J)=(\{1,t\},\emptyset)\)
is a \(W\!\)-graph ideal by case~(vii) of Theorem~\ref{dihedral}, since \(m\) is
even. Since \(\mathscr{I}=\mathscr{I}^{-1}\), it is also a \(W\!\)-graph right ideal.
Observe that \(T_s\) acts as scalar multiplication by~\(q\), in both the left action
and the right action. Moreover, the left action of \(T_t\) is the same as the right
action. So the left and right \(\mathcal{H}\)-actions commute, as required. Case~(iv)
is the same as case~(iii), and cases (v)~and~(vi) are trivial.

It remains to prove that there are no others. So assume that
\((\mathscr{I}\!,\,J,K)\) is a \(W\!\)-graph biideal. Since \(\mathscr{I}\) has to be
an ideal of \((W,\leqslant\lside)\) and of \((W,\leqslant\rside)\) we see that if \(\mathscr{I}\)
contains some element of length~\(l\) then it must contain all \(2l-1\) elements
of length less than~\(l\). So clearly we must have
\(\mathscr{I}=\{\,[..st]_l\mid l\leqslant h\,\}\cup\{\,[..ts]_l\mid l\leqslant k\,\}\) for
some integers \(h\) and \(k\), with either \(h=k\) or \(|h-k|=1\).

Assume first that \(\min(h,k)\geqslant 1\). Then both \(s\) and \(t\) are in \(\mathscr{I}\),
and Remark~\ref{JandK} shows that \(J=K=\emptyset\).
So Proposition~\ref{notbiideal} shows that \(h=k\), and since \((\mathscr{I}\!,\,J)\)
is a \(W\!\)-graph ideal it follows from Theorem~\ref{dihedral} that either
\(\mathscr{I}=W\) or \(k+1\) is a divisor of~\(m\). So the only possibilities
correspond to case~(i) and case~(ii) in the theorem statement.

Obviously \(h=k=0\) gives case~(v) or case~(vi) of the theorem statement. So it remains to
consider the possibilities that \(h=0\) and \(k=1\), giving \(\mathscr{I}=\{1,s\}\),
or \(h=1\) and \(k=0\), giving \(\mathscr{I}=\{1,t\}\). Since \(h+1\) and \(k+1\) have
to be divisors of~\(m\), it follows that \(m\) must be even. If 
\(J=K=\emptyset\) then we obtain cases (iii)~and~(iv) of the theorem statement.
We must show that all other cases lead to contradictions.

Suppose first that \(\mathscr{I}=\{1,s\}\). Then \(s\notin J\) and \(s\notin K\), 
and since \(J\) and \(K\) are not both empty, one or other must be~\(\{t\}\). Let
\(\{c_1,c_s\}\) be the \((W\times W\opp)\)-graph basis of the bimodule
\(\mathscr{S}(\mathscr{I}\!,\,J,K)\). If \(J=\{t\}\) then
\[
(T_tc_1)T_s=(-q^{-1}c_1)T_s=-q^{-1}(qc_1+c_s)\ne -c_1+qc_s=T_t(qc_1+c_s)=T_t(c_1T_s),
\]
while if \(K=\{t\}\) then 
\[
T_s(c_1T_t)=T_s(-q^{-1}c_1)=-q^{-1}(qc_1+c_s)\ne -c_1+qc_s=(qc_1+c_s)T_t=(T_sc_1)T_t.
\]
So in either case we have a contradiction. A similar argument disposes of
\(\mathscr{I}=\{1,t\}\).
\qed
\end{proof}

\end{document}